\documentclass[bj]{imsart}

\RequirePackage{amsthm,amsmath,amsfonts,amssymb}
\RequirePackage[colorlinks,citecolor=blue,urlcolor=blue]{hyperref}
\RequirePackage{graphicx}
\RequirePackage{dsfont}
\RequirePackage{color}
\RequirePackage{tikz}
\usetikzlibrary{matrix,calc}
\RequirePackage{array}
\newcolumntype{P}[1]{>{\centering\arraybackslash}p{#1}}
\newcolumntype{M}[1]{>{\centering\arraybackslash}m{#1}}
\RequirePackage{algorithm}

\usepackage{times}
\usepackage{stmaryrd}

\renewcommand{\leq}{\leqslant}
\renewcommand{\geq}{\geqslant}

\newcommand{\inner}[2]{\langle #1, #2 \rangle}
\newcommand{\innerl}[2]{\left\langle #1, #2 \right\rangle}

\makeatletter
\renewenvironment{cases}[1][l]{\matrix@check\cases\env@cases{#1}}{\endarray\right.}
\def\env@cases#1{%
  \let\@ifnextchar\new@ifnextchar
  \left\lbrace\def\arraystretch{1.2}%
  \array{@{}#1@{\quad}l@{}}}
\makeatother

\usepackage{amsthm}
\usepackage{amsmath}
\usepackage{algorithm}
\usepackage{algpseudocode}
\usepackage{amssymb}
\usepackage{bbm}
\usepackage{mathtools}
\DeclareMathOperator{\E}{\mathbb{E}\,}
\DeclareMathOperator{\rank}{rank}
\DeclareMathOperator{\tr}{tr}
\DeclareMathOperator{\diag}{\textrm{diag}}
\DeclareMathOperator{\Crit}{Crit}

\DeclareMathOperator{\Proj}{\mathcal{P}}

\startlocaldefs
\theoremstyle{plain}

\newtheorem{theorem}{Theorem}[section]
\newtheorem{assumption}{Assumption}[section]
\newtheorem{open problem}{Open Problem}[section]
\newtheorem{corollary}{Corollary}[section]
\newtheorem{proposition}{Proposition}[section]
\newtheorem{lemma}{Lemma}[section]
\numberwithin{equation}{section}

\newtheorem{remark}{Remark}[section]
\newtheorem{definition}[theorem]{Definition}

\endlocaldefs

\begin{document}

\begin{frontmatter}
\title{Efficient learning of hidden state LTI state space models of unknown order}
\runtitle{Learning LTI state space models of unknown order}

\begin{aug}
\author[A]{\inits{B.}\fnms{Boualem} \snm{Djehiche}\ead[label=e1,mark]{boualem@kth.se}}
\and
\author[A]{\inits{O.}\fnms{Othmane} \snm{Mazhar}\ead[label=e2,mark]{othmane@kth.se}}

\address[A]{Department of Mathematics,
KTH Royal Institute of Technology, Stockholm, Sweden.
\printead{e1,e2}}
\end{aug}

\begin{abstract}
The aim of this paper is to address two related estimation problems arising in the setup of hidden state linear time invariant (LTI) state space systems when the dimension of the hidden state is unknown. Namely, the estimation of any finite number of the system's Markov parameters and the estimation of a minimal realization for the system, both from the partial observation of a single trajectory. For both problems, we provide statistical guarantees in the form of various estimation error upper bounds, $\rank$ recovery conditions, and sample complexity estimates.

Specifically, we first show that the low $\rank$ solution of the Hankel penalized least square estimator satisfies an estimation error in $S_p$-norms for $p \in [1,2]$ that captures the effect of the system order better than the existing operator norm upper bound for the simple least square. We then provide a stability analysis for an estimation procedure based on a variant of the Ho-Kalman algorithm that improves both the dependence on the dimension and the least singular value of the Hankel matrix of the Markov parameters. Finally, we propose an estimation algorithm for the minimal realization that uses both the Hankel penalized least square estimator and the Ho-Kalman based estimation procedure and guarantees with high probability that we recover the correct order of the system and satisfies a new fast rate in the $S_2$-norm with a polynomial reduction in the dependence on the dimension and other parameters of the problem.

\end{abstract}



\end{frontmatter}

\section{Introduction}

Many control design and synthesis techniques rely on an accurate description of the system as a state space model. Deriving such accurate description is an important problem in system identification with far-reaching applications in many areas including time series analysis \cite{citeulike:105953}, economics \cite{GVK237148978}, robotics \cite{10.5555/1643720} and aeronautics \cite{JUNKINS2010xiii}, to name a few. While in some cases it is possible to derive such models due to the simple structure of the underlying phenomena involved in the dynamical evolution \cite{Ljung89,DomitillaMurray}, there are many instances where such an approach is intractable because the system is too complex or some of the involved phenomena are not well understood. In those cases, one adopts a so-called black-box approach and learns the system from the input/output data generated in an experimental setting with little to no assumptions on the real system. 

In recent years, there has been an increased interest in providing non-asymptotic statistical guarantees in the form of estimation error upper bounds and sample complexity estimates for data-driven estimation procedures for state space models \cite{774109,1024346,VIDYASAGAR2008421,SHIRANIFARADONBEH2018342,pmlr-v75-simchowitz18a,DBLP:journals/corr/abs-1812-01251}. While there is a plethora of estimation procedures for learning state space models most of which are well understood in the asymptomatic regime derived e.g. in  \cite{Ljung89,980485,10.1016/S0005-1098(99)00174-0,VanOverschee1996}, modern estimation setups present additional challenges that are not taken into account in an asymptomatic study. For example, in estimation based solution to the linear reinforcement learning problem \cite{pmlr-v144-lale21b} one would aim to obtain a highly accurate estimate of the dynamical system as fast as possible before moving to the control part or alternate between estimation and control in such a way to strike a trade-off between the exploration and the exploitation part. In these cases, asymptotic results are of limited use; the more accurate measure of estimation performance would be through the non-asymptotic estimation error and the sample complexity. Beyond the reinforcement learning use, non-asymptotic statistical guarantees are also used in conjunction with robust control techniques \cite{8618658,tu2017nonasymptotic}, in control design using the Markov parameters \cite{411242,FURUTA19951317}, and as theoretical guidelines for practical heuristics such as bootstrapping to establish high-probability confidence intervals \cite{Dean2020}. Several authors provided such estimates for observed LTI state space models, and the results are essentially optimal in the sense that upper and lower bounds for both the estimation error and sample complexity match up to logarithmic terms and unknown multiplicative constants \cite{pmlr-v75-simchowitz18a}. However, the situation is not as clear-cut for Hidden state LTI state space models. In a realistic setup, these systems present the additional challenge of not knowing the dimension of the state. In the absence of precise estimation lower bounds, the estimation upper bounds presented in the literature so far do not capture well the effect of the system dimension and dynamic on the non-asymptotic estimation error and sample complexity incurred by the studied estimation algorithms. 

\subsection{Problem statement and preliminaries} \label{sec: Problem statement and preliminaries}
The present study aims to provide computationally effective estimation procedures satisfying sharp non-asymptotic estimation error bounds and sample complexity estimates when used for learning the parameters of hidden state LTI state space models of unknown order from the partial observation of a single trajectory of the system. We then have to deal with two ambiguities:
\begin{itemize}
    \item The hidden dimension of the parameters is not well defined from an input/output standpoint.
    \item The hidden state LTI state space model parameters' are defined only up to a similarity transform.
\end{itemize}
In this section, we make these two claims more precise, introduce some necessary preliminaries from realization theory and provide a precise statement for the aim of the study.    

To this end, we consider the following formulation for LTI state space models.
 \begin{align} \label{LTI} 
 \begin{cases}[r]
 x_{i+1} &\hspace{-8pt} = A_0x_i + B_0u_i + w_i,\\
 y_i &\hspace{-8pt} = C_0x_i + v_i,
 \end{cases}
 \end{align}
where $x_i \in \mathbb{R}^{d_0}$ is the hidden state variable of unknown dimension $d_0$ and $u_i \in \mathbb{R}^{r}$ are iid multivariate normal sequences $\mathcal{N}(0,\sigma_u^2 I_r)$. They excite the system to generate the output sequence $y_i \in \mathbb{R}^p$. $w_i$ respectively $v_i$ are the iid multivariate normal state noise sequence $\mathcal{N}(0,\sigma_w^2 I_{d_0})$ and the output noise sequence $\mathcal{N}(0,\sigma_v^2 I_p)$, respectively. These centred Gaussian random vectors can be replaced by centred subGausian centred random vectors of appropriate $\psi_2$ norm upper bounds and all the results remain the same. The linear dynamic is then described by the parameters $(A_0,B_0,C_0)$ with $A_0 \in \mathcal{M}_{d_0 \times d_0}(\mathbb{R})$, $B_0 \in \mathcal{M}_{d_0 \times r}(\mathbb{R})$, and $C_0 \in \mathcal{M}_{p \times d_0}(\mathbb{R})$. If we eliminate the state variable $x_i$, we obtain the so-called input/output (I/O) description of the system:
\begin{align}
    y_n &= \sum \limits_{i=0}^{n-1} C_0A_0^{n-1-i}B_0u_i + \sum \limits_{i=1 0}^{n-1} C_0A_0^{n-1-i}w_i + v_{n} \nonumber\\
    &= \sum \limits_{l=t-2T+1}^{t-1} C_0A_0^{t-1-l}Bu_l + \sum \limits_{l=0}^{t-2T} C_0A_0^{t-1-l}B_0u_l +\sum \limits_{l=0}^{t-1} C_0A_0^{t-1-l}w_l + v_{t} \nonumber \\
    &:= g_0X_t + \bar{g}_0 \bar{X}_t + hW_t  + v_t, \label{eq: IO notation}
\end{align}
where, $\bar{N}:= N-2T+1$,
\begin{equation}\label{eq:IO-1 notation}
\begin{array}{lll}
    X_l := [u^*_{l-1},u^*_{l-2},\dots,u^*_{l-2T+1}]^* \in \mathbb{R}^{(2T-1)r}\\
    \bar{X}_l := [u^*_{l-2T},u^*_{l-2T-1},\dots,u^*_0,0,\dots,0]^* \in \mathbb{R}^{\bar{N}r}\\
    W_l := [w^*_{l-1},w^*_{l-2},\dots,w^*_0,0,\dots,0]^* \in \mathbb{R}^{Nd_0},
\end{array}
\end{equation}
and
\begin{equation}\label{eq:IO-2 notation}
\begin{array}{lll}
    g_0 := [C_0B_0,C_0A_0B_0,\dots,C_0A_0^{2T-2}B_0] \in \mathcal{M}_{p\times (2T-1)r }(\mathbb{R})\\
    \bar{g}_0 := [C_0A_0^{2T-1}B_0,C_0A_0^{2T}B_0,\dots,C_0A_0^{N-1}B_0] \in \mathcal{M}_{p\times \bar{N}r }(\mathbb{R})\\
    h := [C_0,C_0A_0,\dots,C_0A_0^{N-1}] \in \mathcal{M}_{p \times Nd_0}(\mathbb{R}).
\end{array}
\end{equation}
Since we want to estimate the parameter $g_0$, the part $\bar{g}_0 \bar{X}_t + hW_t  + v_t$ will play the role of a disturbance that we will refer to as the noise part. To obtain a successful estimator of $g_0$,  this part should not grow arbitrarily large. To this end, we impose an assumption on the growth of the powers of the  estimated matrix in terms of its spectral radius that we recall in the next definition.
\begin{definition}
    The spectral radius of a matrix $A$ is defined as $\rho(A) := \min_{\lambda \in \text{SP}(A)} |\lambda|$, where $SP(A)$ is the spectrum (the set of all eigenvalues) of $A$.
\end{definition}
\begin{assumption} \label{Assumption: Stability}
We assume that the system \eqref{LTI} is stable in the sense that the spectral radius $\rho(A_0)$ is strictly less than $1$.   
\end{assumption}
By applying the Jordan decomposition to the matrix $A$, we readily see that there exists a positive constant $\psi_A$ depending only on $A$ such that for all $k \in \mathbb{N}$ we have 
\begin{equation} \label{eq: relation operator norm and spectral radius}
    |A^k|_{S_\infty}\leq \psi_A\rho(A^k).
\end{equation}
The system identification problem in this setup would be to estimate the parameters $(A_0,B_0,C_0)$ given that we observe a single realization of $(X_i,y_i)_{i=2T}^N$ while we do not have access to the sequence $(x_i)_i$ and in particular we do not know the dimension $d_0$. From \eqref{eq: IO notation} we notice that the sequence $(y_i)_i$ is related to $(u_i)_i$ in a causal fashion only through the factors $(CA^iB)_i$, commonly referred to as the Markov parameters associated with the system $(A_0,B_0,C_0)$.

\medskip We note that for any similarity transform $S$, the parameters $(A_0,B_0,C_0)$ and their transforms  $(SA_0S^{-1}, SB_0,C_0S^{-1})$ give the same values for the Markov parameter vector $g_0$. This makes the problem of learning the parameters $(A_0,B_0,C_0)$ from observations up to time $N$ of a single trajectory $(u_i,y_i)_i$ not well defined. One can only learn a representative of the equivalence class defined by the parameters $(SA_0S^{-1},SB_0,C_0S^{-1})$ for all similarity transforms $S$. This also makes the dimension $d_0$ not well defined as one can always replace the system \eqref{LTI} by the larger system
\begin{align*} 
    \begin{cases}[r]
    \begin{bmatrix}x_{i+1}\\ z_{i+1} \end{bmatrix} &\hspace{-8pt} = \begin{bmatrix}A_0 &\\ \ &\ \end{bmatrix}\begin{bmatrix}x_i\\ z_i \end{bmatrix} + \begin{bmatrix} \\ B_0 \end{bmatrix}u_i + w_i,\\
    y_i &\hspace{-8pt} = \begin{bmatrix} C_0 & \  \end{bmatrix}\begin{bmatrix} x_i\\ z_i \end{bmatrix} + v_i.
    \end{cases}
\end{align*}
Nonetheless, from the Realization Theory of linear systems, we know a representative of the equivalence class for $(A_0,B_0,C_0)$ of minimal dimension exists.
\begin{definition} \label{definition system order}
    We refer to a representative of the equivalence class of minimal dimension as a minimal realization, and we refer to the dimension of the minimal realization as the system order.
\end{definition} 
The system order coincides with the McMillan degree~\cite{Mcmillan52,Mcmillan52b} defined as
\begin{equation*}
    \delta(A_0,B_0,C_0) := \rank(Hg_0^*),
\end{equation*} 
where $H \colon \mathcal{M}_{p \times (2T-1)r}(\mathbb{R}) \longrightarrow \mathcal{M}_{Tp \times Tr}(\mathbb{R})$ is the $T$ order Hankel operator on $g_0$ defined, for any $ g  = [g_1,\dots,g_{2T-1}] \in \mathcal{M}_{p \times (2T-1)r}(\mathbb{R})$, by
\begin{equation*}
    Hg^*= \begin{bmatrix}
    &g_1 &g_2 &g_3  & &g_T \\
    &g_2 &g_3 &  & &g_{T+1} \\
    &g_3 & &  & &g_{T+2} \\
    & & & & & \\
    &g_{T} &g_{T+1} &g_{T+2}  & &g_{2T-1} \\
    \end{bmatrix},
\end{equation*}
and where $T$ is greater than the dimension of the matrix $A_1$ of some particular realization $(A_1,B_1,C_1)$ which is not necessarily minimal. For more on Hankel operators, their properties and the role of the McMillan degree as a complexity measure for LTI models, we refer to \cite{bottcher2012introduction,partington1988introduction,Mcmillan52,Mcmillan52b}. We also note that the McMillan degree is independent of the realization since it is defined with respect to the Markov parameters. 

\medskip Hence, we deal with the ambiguity in the definition of the system dimension by adopting the following 
\begin{assumption}
We assume that the realization $(A_0,B_0,C_0)$ is minimal in the sense that $d_0 =\delta(A_0,B_0,C_0)$.
\end{assumption}
This assumption can be made without loss of generality since any hidden state LTI state space system has a minimal realization. We also define the $T$-order controllabilty matrix $\mathcal{C}$ and the $T$-order observability matrix $\mathcal{O}$ for the realization $(A_0,B_0,C_0)$ by 
\begin{equation}\label{C-O}
    \mathcal{C} =\begin{bmatrix}
        B_0 &A_0B_0 &\cdots&A_0^{T-1}B_0
    \end{bmatrix} \quad \text{and} \quad \mathcal{O} =\begin{bmatrix}
        C_0 \\C_0A_0 \\\vdots\\C_0A_0^{T-1}
    \end{bmatrix}
\end{equation}
and recall that the system $(A_0,B_0,C_0)$ is a minimal realization if and only if $\rank(\mathcal{C}) = \rank(\mathcal{O}) = d_0$ and in that case, for all $T \geq d_0$, we have $\rank(Hg_0^*) = d_0$. When this occurs we say that the pair $(A,B)$ is controllable and the pair $(C,A)$ is observable. Therefore, we impose the following  
\begin{assumption} \label{Assumption: dimension less than its estimate}
We assume that $T \geq d_0$.
\end{assumption}
As mentioned above, this assumption is necessary and sufficient for the Hankel matrix of the Markov parameters to capture the dimension of the minimal realization. Moreover, this assumption is even more relevant when we estimate a Hidden state LTI state space model of unknown order while given a pessimistic upper bound $T$ on $d_0$, which is the high dimension estimation set up in this context.

Since a minimal realization is again defined up to a similarity transform, we introduce here the concept of a \textit{balanced minimal realization} which is a particular minimal realization that one can compute, given the Markov parameters. The procedure of deriving a minimal realization from the description of the Markov parameters is known as the Ho-Kalman algorithm. Suppose that we are given the Markov parameter vector $g_0$ and noting that $Hg_0^* = \mathcal{O}\mathcal{C}$, the Ho-Kalman algorithm starts from the SVD decomposition of the Hankel matrix of the Markov parameters and constructs the particular minimal realization given in the following 
\begin{definition} \label{definition minimal balanced realization}
Assume that $\rank(Hg_0^*) = d_0$ and $T \geq d_0+1$. Then a minimal balanced minimal realization $(\bar{A},\bar{B},\bar{C})$ is defined through the following Ho-Kalman algorithm: 
\begin{itemize}
    \item Define the SVD decomposition of the Hankel matrix of the Markov parameters by 
    \begin{equation*}
        Hg_0^* = U_0 \Sigma_0 V_0^* .   
    \end{equation*}
    \item Take 
    \begin{equation*}
        \bar{O} =U_0 \Sigma_0^{1/2} \ \text{ and } \ \bar{C} =\Sigma_0^{1/2} V_0^*.    
    \end{equation*} 
    \item Define the minimal balanced realization as
    \begin{equation*}
        \bar{A} = \left(\bar{O}_{1:r(T-1),1:d_0}\right)^{\dagger}\bar{O}_{r+1:rT,1:d_0},\ \bar{B} = \bar{O}_{1:d_0,1:r},\ \text{and}\ \bar{C} = \bar{O}_{1:p,1:d_0}.
    \end{equation*}
\end{itemize}
\end{definition}
In this definition, $M^{\dagger}$ refers to the left pseudo inverse of a full column rank matrix $M$ and $M_{a:b,c:d}$ refers to the sub-matrix of $M$ composed of rows $a$ to $b$ and columns $c$ to $d$. 

We note that there are multiple variants of the Ho-Kalman algorithm described in the previous definition, but the main idea for the construction is the same for all of them. 

We note as well that $\bar{O}_{1:r(T-1),1:d_0}$ has full rank since it is equal to the observability matrix up to a similarity transform. Hence, we have
\begin{equation*}
    s_{d_0} (\bar{O}_{1:r(T-1),1:d_0}) > 0,
\end{equation*}
where $s_{k} (M)$ refers to the $k^{\text{th}}$ singular value of the matrix $M$ and the singular values are taken in a decreasing order. Starting from the observation that $Hg_0^* = \mathcal{O}\mathcal{C}$ one can check that there exists a similarity transform $S$ such that $\bar{A} = SA_0S^{-1}$, $\bar{B} = SB_0$, and $\bar{C} = C_0S^{-1}$. Thus $(\bar{A},\bar{B},\bar{C})$ is indeed a minimal realization and belongs to the equivalence class of $\mathcal{M} = (A_0,B_0,C_0)$. 

\medskip Thus, our aim is two folds:
\begin{itemize}
    \item to provide an estimation procedure that given the data generated from the observation of a single trajectory $(X_i,y_i)_{i=2T}^N$ up to time $N$ outputs estimates of the Markov parameters $\hat{g} = [\hat{g}_1,\cdots ,\hat{g}_{2T-2}]$ such that the following loss function 
\begin{equation*}
    \mathcal{L}^H_p(\hat{g},g_0) = |H\hat{g}^* - Hg_0^*|_{S_p}
\end{equation*}
is small with high probability.
\item to provide an estimate $\hat{\mathcal{M}} = (\hat{A},\hat{B},\hat{C})$ for the system $\mathcal{M}_0 = (A_0,B_0,C_0)$ of the same dimension $d_0$ as some minimal realization such that the following loss function with respect to the minimal balanced realisation $\bar{\mathcal{M}} = (\bar{A},\bar{B},\bar{C})$  
\begin{equation*}
    \mathcal{L}^{\mathcal{M}}_p(\hat{\mathcal{M}}, \bar{\mathcal{M}}) := \inf \limits_{S: \ \det(S) \neq 0}  |S^{-1} \hat{A}  S - \bar{A}|_{S_p} + |S^{-1} \hat{\mathcal{B}}  - \bar{B}|_{S_p} + |\hat{\mathcal{C}}  S - \bar{C}|_{S_p}
\end{equation*}
is also small with high probability. Up to a multiplicative constant, this is the same as saying that the loss function $\mathcal{L}(\hat{\mathcal{M}}, \mathcal{M}_0)$ is small.
\end{itemize}
Here, 
$
|M|_{S_p} = (\tr(M^* M)^{p/2})^{1/p},\, 1 \leq p < \infty, \,\, |M|_{S_\infty} = \max \limits_{|x|_2 \leq 1} |Mx|_2$ and $  |x|^2_{2}=\sum_{i=1}^n |x_i|^2$.

\subsection*{Frequently used notation}
Before we review the literature related to our problem and the main contributions of the present paper, we recall some frequently used notations.

We denote by $(\Omega,\mathcal{F},\mathbb{P})$ the underlying probability space and by $\E$ the corresponding expectation operator. 

Here and throughout the paper $c$ denote a positive constant whose exact value is not important for the derivation and might change from one step to another. $x \lesssim y$ is a shorthand for `there exists a positive constant $c$ such that $x \leq cy$', and $x \simeq y$ means that $x \lesssim y$ and $y \lesssim x$. The minimum (maximum) of two real numbers $x$ and $y$ is denoted as $\min(x,y)= x \wedge y$ ($\max(x,y)= x \vee y$).

Whenever possible, our results are provided with explicit constants to give an idea of their order. The numerical values of these constants are useful in practice but are not optimal and can be improved.

\subsection{Related literature} \label{sec: related literature}
A common estimation approach in the Hidden state LTI state space setup is the two-step approach commonly referred to as a subspace method \cite{VanOverschee1996,verhaegen_verdult_2007,doi:10.1080/00207178408933239}. In the first step of this approach,  one learns the Markov parameters with a good enough precision, since unlike the true parameters $(A_0,B_0,C_0)$ the Markov parameters are well defined, and in the second step, one uses the learned Markov parameters to provide an estimate close to a representative of the equivalence class of the true parameters. The first step is usually carried out with a regression-type estimator and the second step is carried out via some variant of the celebrated Ho-Kalman algorithm, which relies on identifying a possible realization from the output of an SVD decomposition. The popularity of the subspace approach is because it is computationally tractable, unlike the maximum likelihood approach or the predictive error method, which both results in a non-convex optimization problem \cite{Ljung89}. Several results appeared recently in the machine learning community studying non-asymptotic properties of variants of the subspace method under various assumptions on the estimation setup. The literature on the estimation of the parameters of LTI state space models is very rich; early works in the topic date back to the nineties where \cite{Deistler1995ConsistencyAR,PETERNELL1996161,VIBERG19971603,KNUDSEN200181} provided asymptotic results. A complete overview of this vast literature falls beyond the format and the scope of the present paper. Therefore, we only mention and discuss here some recent results \cite{9440770,JMLR:v22:19-725,Jian16,pmlr-v120-sun20a} that provide non-asymptotic statistical guarantees for variants of the subspace method and thus are close in spirit to our work.

\begin{remark}
We took some freedom to omit the contribution of lower order terms for some of these results. Instead, we refer to the original work for the exact statement.
\end{remark}

For ease of notation, we set
\begin{equation} \label{eq: short hand matrix notation for LTI} \begin{array}{lll}
    y := \begin{bmatrix}y_{2T}^*\\
    y_{2T+1}^*\\
    \vdots \\
    y_N^*
    \end{bmatrix},\ 
    X := \begin{bmatrix}X_{2T}^*\\
    X_{2T+1}^*\\
    \vdots \\
    X_N^*
    \end{bmatrix}, \ 
    \bar{X} := \begin{bmatrix}\bar{X}_{2T}^*\\
    \bar{X}_{2T+1}^*\\
    \vdots \\
    \bar{X}_N^*
    \end{bmatrix}, \\  
    W := \begin{bmatrix}W_{2T}^*\\
    W_{2T+1}^*\\
    \vdots \\
    W_N^*
    \end{bmatrix},
    \,\, \;
    \varepsilon = \begin{bmatrix}
    v_{2T}^*\\
    v_{2T+1}^*\\
    \vdots \\
    v_N^*
    \end{bmatrix}.
    \end{array}
\end{equation}
Thus, we can write the input/output representation \eqref{eq: IO notation} for the $y$ vector more succinctly  as follows:
\begin{equation*}
    y = Xg_0^* + \bar{X}\bar{g}_0^* + Wh^* + \varepsilon.
\end{equation*}
\begin{itemize}
    \item \textit{The context of known dimension $d_0$}. While this context is simpler, results in this setup are informative about what can be expected if $d_0$ is unknown. 
    Oymak and Ozay~\cite{9440770} consider a subspace approach in this context and show that the least square estimator defined as $ \hat{g}_{\text{ls}} := (X^{\dagger}y)^* $ can effectively learn the first $T$ Hankel parameters in the sense that with high probability and for values of $N$ such that
    \begin{equation*}
        N \geq N_0 = cTq_0 \log^2(Tq_0)\log^2(TN) \quad \text{with} \quad q_0 = r+p+d_0,
    \end{equation*}
    it holds that \cite[Theorem ~$3.1$]{9440770}
    \begin{equation*}
        |\hat{g}_{\text{ls}} - g_0|_{S_\infty}\leq (\sigma_v+ \sigma_e+|h|_{\mathcal{H}_\infty}\log(TN) )\sqrt{\frac{cTq_0 \log^2(Tq_0)}{N}},
    \end{equation*}
    where $\sigma_e$ accounts for the variance of $x_{t-T}$. Under the same condition it was shown that a version of the Ho-Kalman algorithm successfully learns, up to a similarity transform, a representation of the true parameter on the same event for $N \geq \frac{N_0}{s^2_{d_0}(Hg_0)}$ with the guarantee of \cite[Theorem ~$5.3$]{9440770}
    \begin{equation*}
        \mathcal{L}^{\mathcal{M}}_2(\hat{\mathcal{M}}, \bar{\mathcal{M}}) \lesssim \frac{\left(\sigma_v+ \sigma_e+|h|_{\mathcal{H}_\infty}\log(TN)\right)|Hg_0|_{S_\infty}q_0\sqrt{T \log^2(Tq_0)} }{s^2_{d_0}(Hg_0)\sqrt{N}}.
    \end{equation*}
    \item \textit{The context of unknown dimension $d_0$}. Sarkar et al.~\cite{JMLR:v22:19-725} adopt a model selection approach to choose a realization of order $\hat{d}$ that is good enough. Their learning algorithm proceed in three stages:
    \begin{enumerate}
        \item Hankel matrix estimation: the algorithm starts by solving, for all $d \in \mathcal{D}(N)= \{ T\ | \ N \geq Tr^2\log^3(Nr/\delta)\}$, a least square problem to get an estimated Hankel matrix of $2T+1$ parameter $\hat{H}_T$.
        \item Order selection:  the algorithm chooses a model of size $\hat{d}$ according to the rule
        \begin{equation*}
            \hat{d} = \tilde{d}\wedge \log(N/\delta),
        \end{equation*}
        with $\alpha(h) = \sqrt{\frac{hp+h^2r + \log(N/\delta)}{N}}$,
        where
        $$
        \tilde{d}:=\inf\{d \in \mathcal{D}(N); \ |\hat{H}_d - \hat{H}_l|_{S_\infty} \leq c(\alpha(d) + \alpha(l)) \ \forall l \geq d, \ l \in \mathcal{D}(N) \}.
        $$
    \item Parameter estimation: the algorithm uses a variant of the Ho-Kalman algorithm to get a realization $(\hat{A},\hat{B},\hat{C})$ of dimension $\hat{d}$ from $\hat{H}_{\hat{d}}$.
    \end{enumerate}
    Their results \cite[Theorem ~$5.1$ and Proposition ~$5.1$]{JMLR:v22:19-725} imply that, for all $T \in \mathcal{D}(N)$ and 
    \begin{equation*}
        N \geq c(r^2T \log^2(T)\log^2(r/\delta) +T\log^2(T)),
    \end{equation*}
    the estimation step outputs an estimate for the Hankel matrix of the parameters satisfying
    \begin{equation} \label{eq: for comparison lp bounds}
        |\hat{\mathcal{H}}_T - Hg_0^*|_{S_\infty}\leq c\sqrt{\frac{pT^2 +rT+T\log(1/\delta) }{N}}.
    \end{equation}
    They also show \cite[Theorem ~$5.3$]{JMLR:v22:19-725} that a variant of the Ho-Kalman applied to the selected model $\hat{d}$ successfully learns the best $\hat{d}$ approximation to the minimal realization after observing $N \geq N_*$ sample and we have with probability at least $1-\delta$ the following
    \begin{multline} \label{eq: for comparison parametric bounds}
        \mathcal{L}^{\mathcal{M}}_{\infty}(\hat{\mathcal{M}}, \bar{\mathcal{M}}_{\hat{d}}) \leq \varepsilon \Gamma (\hat{H},\varepsilon) + \frac{\varepsilon \hat{d}}{\sqrt{s_{\hat{d}}(\hat{H})}} \wedge \sqrt{\varepsilon \hat{d}} + \frac{\varepsilon}{\sqrt{s_{\hat{d}}(\hat{H})}} \wedge \sqrt{\varepsilon}, 
        \end{multline}
        \begin{multline*}\\
        \text{with} \quad \varepsilon = c\sqrt{\frac{r \hat{d} + p \hat{d}^2 + \hat{d} \log(N/\delta)}{N}},\ N_* < \infty, \ \text{and} \ \Gamma (\hat{H},\varepsilon) < \infty.
    \end{multline*}
    Here $\bar{\mathcal{M}}_{\hat{d}}$ is the model resulting from the use of the Ho-Kalman algorithm on the truncated SVD of $Hg_0^*$ to the first $\hat{d}$ singular values.  
    \item \textit{The context of unknown dimension $d_0$ while allowing the partial observation of $N$ paths $\left((u_i^j,y_i^j)_{i=1}^{2T-1}\right)_{j=1}^N$ of length $2T-1$  without process noise:} This setup is different from ours as it allows  multiple independent realizations and assumes that $w_i=0$ which is the main source of difficulty in our setup, nonetheless the approaches used in this context in \cite{Jian16,pmlr-v120-sun20a} are closer to our approach as they relay on restricted or penalized least square estimators to estimate the Markov parameters. Indeed, \cite{Jian16} analyzes the performance of the following estimator $\hat{g}_{\text{rls}}$ in the problem of robust recovery of a superposition of distinct complex exponential functions from few random Gaussian projections.
    \begin{equation*}\begin{array}{lll}
        \hat{g}_{\text{rls}} \in \arg \min_{g} \,\, |Hg^*|_{S_1}\\
        \quad \textrm{s.t.} \,\, |XKg^* - y|_2 \leq \delta,
        \end{array}
    \end{equation*}
    with $K=\diag(\sqrt{1},\cdots,\sqrt{T},\sqrt{T-1},\cdots,\sqrt{1})$. This problem is indeed equivalent to the estimation problem of single input single output LTI state space models from multiple trajectories with weighting $K$ for the input without process noise. They show that, with probability at least $1-e^{-cN}$ for $N \gtrsim d_0\log^2(T) + \varepsilon$, the following holds
    \begin{equation*}
        |Kg_0^* - K\hat{g}^*_{\text{rls}}|_2 \leq c\frac{\delta}{\varepsilon}.
    \end{equation*}
    Inspired by this result, Sun {\it et al.} \cite{pmlr-v120-sun20a} use the following nuclear norm penalized least square estimator $\hat{g}_{\text{pls}}$ for the multiple input single output case
    \begin{align} \label{eq: another penalized regressor}
        \hat{g}_{\text{pls}} \in \arg \min_{g} \quad & |XK^{-1}g^* - y|_2 + \lambda|HK^{-1}g^*|_{S_1},
    \end{align}
    and show that with high probability, for a choice of $\lambda = \frac{T\sigma_z}{\sigma_u}\sqrt{\frac{r}{N}} \log(T)$,
    \begin{equation} \label{eq: for comparison lp bounds unknown dimension}
        \mathcal{L}^\mathcal{H}_\infty(\hat{g}_{\text{pls}},g_0) \lesssim \begin{cases}
        \frac{\sigma_z}{\sigma_u} \sqrt{\frac{rT^2}{N}} \log(T),& N \geq d_0^2\wedge T,\\
        \frac{\sigma_z}{\sigma_u} \sqrt{\frac{d_0rT^2}{N}} \log(T), & d_0 \leq N \leq d_0^2\wedge T.
        \end{cases}
    \end{equation}
    
\end{itemize}
\subsection{Main contributions}

As mentioned in Section \ref{sec: Problem statement and preliminaries}, we consider the parametric estimation task in the setup of LTI state space model \eqref{LTI} from the observation of a single trajectory when neither the state is observed nor the system's order is known. In what follows, we present our contributions. 

\begin{remark} \label{remark: regime for upper bound comparison}
While some of the results presented above are provided in terms of the norm $|\cdot|_{S_\infty}$, ours are derived for the norm $|\cdot|_{S_p}$ with $p \in [1,2]$. Whenever it is the case, we use the norm domination relation relation $|\cdot|_{S_p} \leq r^{1/p}|\cdot|_{S_\infty}$ for the sake of comparison, where $r$ is the appropriate dimension.  

From the related literature we see that up to logarithmic terms all the upper bounds are of the form $P(d_0,T,s_{d_0}^{-1}(\bar{\mathcal{O}}^+),s_{d_0}^{-1}(Hg_0^*),N^{-1})$ where $P$ is some polynomial function of these variable. All throughout, we compare different results in the asymptotic regime where $N \to \infty$, $d \to \infty$, $T \to \infty$, $s_{d_0}(Hg_0^*) \to 0$, and $s_{d_0}(\bar{\mathcal{O}}^+)\to 0$ while the upper bound still converge to $0$.
\end{remark}

In Section \ref{sec: Hankel penalized regression}, we provide non-asymptotic estimation error upper-bounds and sample complexity for 
the Hankel penalized regression estimator given by any particular solution of the convex optimization problem
\begin{equation}\label{g-hat}
    \hat{g} \in \arg \min \limits_{g \in \mathcal{M}_{p\times (2T-1)r}(\mathbb{R})} \frac{1}{\bar{N}} |y-Xg^*|_{S_2}^2 + \lambda |Hg^*|_{S_1}.
\end{equation}
For this estimator we provide in Theorem \ref{thm: Performance of the Hankel penalized regression estimator} estimation guarantees and sample complexity for different dimension sensitive loss functions. In particular, we show with probability at least $1-\delta$ and for $\bar{N}$ large enough, that the $p$-loss function $\mathcal{L}^H_p(\hat{g},g_0)$, for $p \in (0, 1)$, satisfies
\begin{align*}
    \mathcal{L}^H_p(\hat{g},g_0) &\lesssim d_0^{1/p} T\sqrt{\frac{p+r + \log(T/\delta) }{\bar{N}}}.
\end{align*}
The available upper bounds for these loss functions are derived in the case $p = \infty$ for the least square estimator when the dimension $d_0$ is known; see for instance \cite[Theorem ~$5.1$]{JMLR:v22:19-725}. Since the solution of the least square estimator is not low $\rank$, the estimate \eqref{eq: for comparison lp bounds} implies
\begin{equation*}
    |\hat{\mathcal{H}}_d - Hg_0^*|_{S_p}\lesssim T^{1/p}\sqrt{\frac{pT^2 +rT+T\log(1/\delta) }{\bar{N}}}
\end{equation*}
with a suboptimal factor $T^{1/p}$ which is the best one would hope for from a non-low $\rank$ estimation procedure. In the same fashion, our result is an improvement of the result \eqref{eq: for comparison lp bounds unknown dimension} with unknown order, while observing multiple trajectories. We finally show in Proposition \ref{prop: Fast rate and rank recovery via SVD} how we can recover the system order efficiently using a truncated SVD procedure if a lower bound on $s_{d_0} (\bar{O}_{1:r(T-1),1:d_0})$ is known. We refer to the discussion after Theorem \ref{thm: Performance of the Hankel penalized regression estimator} for more on this issue. 

In Section \ref{sec: Ho-Kalman} we provide a robustness analysis in the norm $|\cdot|_{S_2}$ of an estimation procedure for the parameters based on a variant of the Ho-Kalman algorithm. In Theorem \ref{thm: stability of Ho-kalman} we show that under some stability conditions, it is possible to recover the parameters if we reduce the error term $|\hat{\mathcal{H}}_{\check{d}_{\xi}} -H g_0^*|_{S_2}$ since 
        \begin{equation*}
            \mathcal{L}^{\mathcal{M}}_2(\hat{\mathcal{M}}, \bar{\mathcal{M}}) \lesssim  \frac{ \left|\bar{A} \right|_{S_\infty}|\hat{\mathcal{H}}_{\check{d}_{\xi}} -H g_0^*|_{S_2} }{s_{d_0}(\bar{\mathcal{O}}^+)s^{1/2}_{d_0}(Hg_0^*)}  .
        \end{equation*} 
This is an improvement of \cite[Theorem ~$4$]{tsiamis2019finite} which  gives 
        \begin{equation*}
            \mathcal{L}^{\mathcal{M}}_2(\hat{\mathcal{M}}, \bar{\mathcal{M}}) \lesssim  \frac{ d_0\left|Hg_0^* \right|^{1/2}_{S_\infty} |\hat{\mathcal{H}}_{\check{d}_{\xi}} -H g_0^*|_{S_\infty} }{s_{d_0}^2(\bar{\mathcal{O}}^+)s^{1/2}_{d_0}(Hg_0^*)}  .
        \end{equation*}
and of \cite[Theorem ~$5.2$]{pmlr-v120-sun20a} which yields
        \begin{equation*}
            \mathcal{L}^{\mathcal{M}}_2(\hat{\mathcal{M}}, \bar{\mathcal{M}}) \lesssim  \frac{ d_0^{1/2}\left|Hg_0^* \right|_{S_\infty} |\hat{\mathcal{H}}_{\check{d}_{\xi}} -H g_0^*|_{S_\infty} }{s^2_{d_0}(Hg_0^*)}  .
        \end{equation*}
 We refer to the discussions after Theorem \ref{thm: stability of Ho-kalman} for more on this.

In Section \ref{sec: Non-assumptotic guarantees for algorithm} we provide non-asymptotic estimation guarantees for Algorithm \ref{algo 1} introduced in Section \ref{sec: Algorithm}. The algorithm yields the estimates $\hat{\mathcal{M}} = (\hat{A},\hat{B},\hat{C})$ for  the minimal balanced realisation $\bar{\mathcal{M}} = (\bar{A},\bar{B},\bar{C})$ with probability $1-\delta$ of the same dimension as a minimal realization $d_0$ after observing 
\begin{equation*}
    \bar{N} \geq c d_0 TN_0 \vee T_0\log{\frac{1}{\delta}} \vee \frac{\phi^2 d_0 T_0^2}{\xi^2} \left(N_0 \vee \log{\frac{1}{\delta}} \right) \vee \frac{\phi d_0^{1/2} T_0\log(T_0)}{\xi} \left(N_0 \vee \log{\frac{1}{\delta}} \right)
\end{equation*}
such that
\begin{equation*}
    \mathcal{L}^{\mathcal{M}}_2(\hat{\mathcal{M}}, \bar{\mathcal{M}}) \lesssim \frac{\phi |\bar{A}|_{S_\infty}d_0^{3/2}}{s^2_{d_0}(\bar{\mathcal{O}}^+)}\left( \sqrt{\frac{N_0}{N}} \vee \frac{\log(d_0)N_0}{N} \vee \sqrt{ \frac{\log{\frac{1}{\delta}}}{N }} \vee  \frac{\log(d_0)\log{\frac{1}{\delta}}}{N} \right) .
\end{equation*}
As mentioned in Section \ref{sec: related literature}, to the best of our knowledge, the only available result in our setup is \eqref{eq: for comparison parametric bounds}
obtained by Sarkar {\it et al.} \cite[Theorem ~$5.3$]{JMLR:v22:19-725}. If we disregard the spectral properties of $H\hat{g}^*$ and multiply by $\hat{d}^{1/2}$ to account for the difference of norms, the dominant term in that expression is 
\begin{equation*} 
    \sqrt{\frac{r \hat{d}^4 + p \hat{d}^5 + \hat{d}^4 \log(N/\delta)}{s_{\hat{d}}(\hat{H})N}} .
\end{equation*}
Hence, our result improves the bound \eqref{eq: for comparison parametric bounds} since it provides an upper bound in terms of the actual dimension $d_0$ and not the estimated dimension $\hat{d}$. Also, it reduces the dependence on the dimension by as much as $d_0$. 

\section{Main results} \label{Main results}
 
\subsection{Algorithmic details} \label{sec: Algorithm}

\begin{algorithm}[!t]\caption{Truncated Hankel Penalized Regression with Ho-Kalman State-Space Realization.}\label{algo 1}
\begin{algorithmic}
[1]
    \State \textbf{Compute}:$\hat{A},\ \hat{B},\ \hat{C}$
    \State \textbf{Input}: $(X_i,y_i)_{i=2T_0}^N$, $\lambda_0$, $\xi$
    \State \textbf{Procedure}: Hankel penalized regression
    \State $\qquad \qquad 
                \hat{g} \in \arg \min \limits_{g \in \mathcal{M}_{p\times (2T_0-1)r}(\mathbb{R})} \frac{1}{N} |y-Xg^*|_{S_2}^2 + \lambda_0 |Hg^*|_{S_1}
        $ 
    \State \textbf{Return}: $H\hat{g}^*$\State $y \gets 1$
    \State \textbf{Compute}:$\hat{A},\ \hat{B},\ \hat{C}$
    \State \textbf{Input}: $(X_i,y_i)_{i=2T_0}^N$, $\lambda_0$, $\xi$
    \State \textbf{Procedure}: Hankel penalized regression
    \State $\qquad \qquad 
                \hat{g} \in \arg \min \limits_{g \in \mathcal{M}_{p\times (2T_0-1)r}(\mathbb{R})} \frac{1}{N} |y-Xg^*|_{S_2}^2 + \lambda_0 |Hg^*|_{S_1}
        $ 
    \State \textbf{Return}: $H\hat{g}^*$
    \State \textbf{Procedure}: Order estimation
    \State $\qquad \qquad \check{d}_{\xi} = \sum \limits_{i=1}^{\rank(H\hat{g}^*
    )} \mathbf{1}\{s_i(H\hat{g}^*) \geq 2\xi \} $
    \State \textbf{Return}: $\check{d}_{\xi}$
    \State \textbf{Procedure}: Reduced order Hankel penalized regression
    \State $\qquad \qquad 
                \hat{g}_\xi \in \arg \min \limits_{g \in \mathcal{M}_{p\times (2\check{d}_{\xi} + 1)r}(\mathbb{R})} \frac{1}{N} |y-Xg^*|_{S_2}^2 + \lambda_1 |Hg^*|_{S_1}
        $ 
    \State \textbf{Return}: $H\hat{g}_\xi^*$
    \State \textbf{Procedure}: Reduced order Ho-Kalman Algorithm
    \State $\qquad \qquad  \hat{U}_\xi \hat{\Sigma}_\xi \hat{V}_\xi^* = \text{SVD}(\hat{\mathcal{H}}_{\check{d}_{\xi}}) \quad \text{and} \quad \hat{\mathcal{H}}_{\check{d}_{\xi}} = \sum \limits_{i=1}^{\check{d}_{\xi} + 1}
     s_i(H\hat{g}_\xi^*)\mathbf{1}\{s_i(H\hat{g}_\xi^*) \geq 2\xi \} \hat{u}_i \hat{v}_i^*$
    \State $\qquad \qquad  \hat{O} =\hat{U}_\xi \hat{\Sigma}_\xi^{1/2} \ \text{ and } \ \hat{C} =\hat{\Sigma}_\xi^{1/2} V_\xi^*$    
    \State $\qquad \qquad  \hat{A} = \left(\hat{O}_{1:r(T-1),1:d_\xi}\right)^{\dagger}\hat{O}_{r+1:rT,1:d_\xi},\ \hat{B} = \hat{O}_{1:d_\xi,1:r},\ \text{and}\ \hat{C} = \hat{O}_{1:p,1:d_\xi}.$
    \State \textbf{Return}: $\hat{A},\ \hat{B},\ \hat{C}$
\end{algorithmic}
\end{algorithm}

We start first by describing our Learning Algorithm \ref{algo 1}. The algorithm starts with the 'Hankel penalized regression' step. In this step, it computes a penalized least square estimate for the first $2T_0-1$ Markov parameters by solving the optimization problem
\begin{equation} \label{eq: Hankel penalized regression 1}
    \hat{g} \in  \arg \min \limits_{g \in \mathcal{M}(p\times (2T_0-1)r)(\mathbb{R})} \frac{1}{\bar{N}} |y-Xg^*|_{S_2}^2 + \lambda_0 |Hg^*|_{S_1}.
\end{equation}
The least square part is the fitting term that ensures fidelity to the data; the penalty part ensures the simplicity of the chosen model.
As described in the introduction, a good measure of the model's complexity for hidden state LTI state space models is the rank of the corresponding Hankel operator since it agrees with the system order as given in Definition \ref{definition system order}. The penalty term using the nuclear norm of the Hankel operator ensures that the solution to the optimization problem described in \eqref{eq: Hankel penalized regression 1} has a low Hankel rank as it is the convex relaxation of the rank function. Thus, we would expect that via a good choice of the free parameter $\lambda$ we obtain a good enough, yet simple, model in the sense that it is close to $Hg_0^*$ with $|H\hat{g}^* - Hg_0^* |_{S_2}$ being small and has $\rank(H\hat{g}^* )$ small enough.

The second step of our learning algorithm is 'Order estimation' in which we compute an estimate $\check{d}_{\varepsilon}$ of the true dimension. If in the last step we have made the error $|H\hat{g}^* - Hg_0^*|_{S_2}$ small compared to $s_{d_0}(Hg_0^*)$, the smallest singular value of $Hg_0^*$, on the one hand, for $1 \leq i \leq d_0$ the  singular values $s_i(H\hat{g}^*)$ will be close to the  singular values $s_i(Hg_0^*)$ and on the other hand the singular values $s_i(H\hat{g}^*)$ for $i > d_0$ will be small so that they are well separated from the others. Thus, via an appropriate choice of $\xi$, we can successfully ensure that $\check{d}_{\xi} = d_0$ with high probability.

The third step, 'Reduced order Hankel penalized regression', is similar to the first step except that it aims at estimating the first $2\check{d}_{\xi} + 1$ Markov parameters instead of the $2T-1$ parameters. For this, it solves the following Hankel penalized regression problem:
\begin{equation} \label{eq: Reduced order Hankel penalized regression 2}
    \hat{g}_\xi \in \arg \min \limits_{g \in \mathcal{M}_{p\times (2\check{d}_{\xi} + 1)r}(\mathbb{R})} \frac{1}{\bar{N}} |y-Xg^*|_{S_2}^2 + \lambda_1 |Hg^*|_{S_1}.
\end{equation}
This is done to obtain a more accurate estimate on these first $2\check{d}_{\xi} + 1$ Markov parameters, since they are the only parameters needed for our estimation procedure based on the Ho-Kalman algorithm to get an accurate estimate for the minimal balanced minimal realization $(\bar{A},\bar{B},\bar{C})$.

The last part of our learning algorithm, 'Reduced order Ho-Kalman Algorithm', uses the previous estimate $\hat{g}_\xi$. It starts with a truncated SVD of the Hankel matrix of the estimated $2\check{d}_{\varepsilon} + 1$ Markov parameters from the previous part. Doing this ensures that the $\rank$ of the truncation result $\hat{\mathcal{H}}_{d_{\varepsilon}}$ is the same as the order of the minimal realization with high probability and that the truncation is close enough to the true model in the sense that $|\hat{\mathcal{H}}_{\check{d}_{\varepsilon}} - Hg_0^*|_{S_2}$ is small. This means that the eigenvalues and eigenvectors of both $\hat{\mathcal{H}}_{\check{d}_{\varepsilon}}$ and $Hg_0^*$ are close to each other. Then, it proceeds with getting estimates $\hat{A},\ \hat{B},\ \hat{C}$ using the Ho-Kalman Algorithm steps described in Definition \ref{definition minimal balanced realization}.

Crucial to the success of our algorithm \ref{algo 1} are the three choices of the free parameters: $T_0$ and $\lambda_0$ in the step 'Hankel penalized regression' and $\xi$ in the step 'Order estimation'. In section \ref{sec: Non-assumptotic guarantees for algorithm} we provide values for these free parameters to ensure the high probability of success of Algorithm \ref{algo 1} in both the order recovery task and the estimation task. We also discuss how reasonable the assumption of knowing each of these parameters is and provide the value for the internal variable $\lambda_1$ necessary for the 'Reduced order Hankel penalized regression' step. In the next section, we provide the probabilistic estimates instrumental to those choices.

\subsection{Probabilistic results} \label{sec: section on Probabilistic results}
We first show that the covariance matrix of the covariates
$$
X_l:=[u^*_{l-1},u^*_{l-2},\dots,u^*_{l-2T+1}]^* \in \mathbb{R}^{(2T-1)r}
$$ 
generated along the path of the input of the LTI state space model \eqref{LTI} concentrate around the identity matrix. This is the main content of the following theorem, which is an extension of \cite[Thoerem~$3.4$]{10.3150/20-BEJ1262} to the multidimensional case. Its proof is given in Appendix \ref{Appendix: Proofs of the main probabilistic results}.
\begin{theorem}\label{thm: Restricted eigenvalue property}
If $X_1, \dots, X_N$ are the time shifted covariates of an LTI hidden state space model \eqref{LTI} where the components $(u_i)_i$ are independent centred standardized multivariate Gaussian $\mathcal{N}(0,\sigma_u^2I_r)$ or subGaussian centred random vectors of subGaussian components having the same $\psi_2$ norm upper bound of $\sigma_u$, $X$ is given by \eqref{eq: short hand matrix notation for LTI}. Then, with probability at least $1-\exp{(-t)}$ for $t \geq 1$, it holds that
\begin{equation}
        \left|\frac{1}{\bar{N}}X^*X-\sigma_u^2I_{(2T-1)r}\right|_{S_\infty} \leq c \sigma_u^2\left(\sqrt{\frac{TN_1}{\bar{N}}}
       + \frac{TN_1}{\bar{N}} + \sqrt{\frac{tT}{\bar{N}} } + \frac{tT}{\bar{N}} \right), \label{eq: Isometric property}
\end{equation}
with $N_1 = \log(T) + r$. Under the same conditions, for $\delta \in (0, e^{-1})$, with probability $1-\delta$ and for values of $N$ such that
\begin{equation*}
    \bar{N} \geq  c\left(TN_1 \vee T\log{\frac{1}{\delta}} \right),
\end{equation*}
we have, for all $g\in \mathcal{M}_{p \times (2T-1)r}(\mathbb{R})$, 
\begin{equation*}
    \frac{\sigma_u^2}{2}|g|_{S_2}^2 \leq \frac{1}{\bar{N}}|Xg^*|_{S_2}^2 \leq \frac{3\sigma_u^2}{2}|g|_{S_2}^2.
\end{equation*}
\end{theorem} 
Concentration results for matrices with independent covariates are obtained in \cite{ADAMCZAK2011195} where it is shown that with high probability the following holds
\begin{equation*} 
        \left| \sum \limits_{i=1}^{N-1} x_ix_i^* - \sigma_u^2I_T \right|_{S_\infty} \lesssim  \sigma_u^2 \left(\frac{T}{N} + \sqrt{\frac{T}{N} }  + \frac{d}{N}t + \sqrt{\frac{d}{N} } t^{1/2} \right).
\end{equation*}
Our result, as a multidimensional extension of ~\cite[Theorem ~$3.4$]{10.3150/20-BEJ1262}, shows that a similar result holds for block Toplitz matrices up to the $N_1 = \log(T) + r$ factor appearing in \eqref{eq: Isometric property}. Comparing with the matrix of the covariates of the hidden dynamical system \eqref{LTI}, it is shown in ~\cite[Proposition ~$2.1$]{djehiche2019finite} that a re-scaled version of it does concentrate around the identity if the eigenvalues of the matrix $A_0$ are not on the unit circle, but would fail otherwise.

\medskip
Before stating the second important probabilistic estimate, we introduce the operator ${H^{\dagger}}^{*} \colon \mathcal{M}_{ (2T-1)r  \times p}(\mathbb{R}) \to \mathcal{M}_{Tp \times Tr}(\mathbb{R})$ defined, for any $h  = [h_1^*,\dots,h_{2T-1}^*]^* \in \mathcal{M}_{ (2T-1)r \times p}(\mathbb{R})$, by
\begin{equation} \label{Adjoint}
    {H^{\dagger}}^{*}h= \begin{bmatrix}
    &h_1 &\frac{1}{2}h_2 &\frac{1}{3}g_3  & &\frac{1}{T}h_T \\
    &\frac{1}{2}h_2 &\frac{1}{3}h_3 &  & &\frac{1}{T-1}h_{T+1} \\
    &\frac{1}{3}h_3 & &  & &\frac{1}{T-2}h_{T+2} \\
    & & & & & \\
    &\frac{1}{T}h_{T} &\frac{1}{T-1}g_{T+1} &\frac{1}{T-2}h_{T+2}  & &h_{2T-1} \\
    \end{bmatrix}.
\end{equation} 
It is easy to check that $\inner{h}{g^*} = \inner{{H^{\dagger}}^{*}h}{Hg^*}$ so that the operator ${H^{\dagger}}^{*}$ is the adjoint of the pseudo-inverse of $H$.

We also introduce the $\mathcal{H}_\infty$ norm for an infinite sequence of matrices $\varphi = [\varphi_0,\varphi_1,\dots]$ given by 
\begin{equation*}
    |\varphi|_{\mathcal{H}_\infty} := \sup \limits_{x \in [0\ 1]} |\sum \limits_{j = 0}^{\infty}\varphi_j e^{i2\pi x}|_{S_\infty}.
\end{equation*}
This norm relates to the notion of system norm used in control theory and turns out to be the right measure of how the hidden dynamic impacts the estimation error through the variance. The next theorem supports this claim by providing a control over the noise level induced by the term $\bar{g}_0 \bar{X}_t + hW_t  + v_t$ given in \eqref{eq: IO notation}. 

\begin{theorem}\label{thm: stochatic control of the noise factor}
Assume the random matrices $X$, $\bar{X}$, $W$, and $\varepsilon$ as defined in \eqref{eq: short hand matrix notation for LTI} are generated by running the LTI hidden state space model \eqref{LTI} under either Gaussian or subGaussian noise condition. Define  $g$, $\hat{g}$, $h$ as in \eqref{eq: IO notation}. Then, with probability at least $1-\exp{(-t)}$ for $t \geq 1$, the following bounds for different parts of the noise term  hold:
\begin{equation} \label{eq: first noise term}
       |{H^{\dagger}}^{*}X^*\bar{X}\bar{g}|_{S_\infty} \lesssim \sigma_u^2|\bar{g}|_{\mathcal{H}_\infty} \left( \sqrt{\frac{N_0}{\bar{N}}} + \frac{N_0\log(T)}{\bar{N}} + \sqrt{\frac{t}{\bar{N} } }+  \frac{\log(T)t}{\bar{N}} \right).
\end{equation}
with $N_0 = \log(T) + p + r$.
\begin{equation} \label{eq: second noise term}
       |{H^{\dagger}}^{*}X^*Wh|_{S_\infty} \lesssim \sigma_u\sigma_w|h|_{\mathcal{H}_\infty} \left( \sqrt{\frac{N_2}{\bar{N}}} + \frac{N_2\log(T)}{\bar{N}} + \sqrt{\frac{t}{\bar{N}}} +  \frac{\log(T)t}{\bar{N}} \right).
\end{equation}
with $N_2 = \log(T) + p$
\begin{equation} \label{eq: third noise term}
       |{H^{\dagger}}^{*}X^*\varepsilon|_{S_\infty} \lesssim  \sigma_v^2 \left(\sqrt{\frac{N_2}{\bar{N}}} + \frac{N_2\log(T)}{\bar{N}} + \sqrt{\frac{t}{\bar{N} } }+  \frac{\log(T)t}{\bar{N}} \right).
\end{equation}
\end{theorem}

\subsection{Estimation guarantees for the Hankel penalized regression} \label{sec: Hankel penalized regression}
This section is devoted to the analysis of the performance of the Hankel penalized regression estimator given by \eqref{g-hat}.
This estimator plays a central role in Algorithm \ref{algo 1} since it is used twice. The first time it uses covariates of length $2T_0-1$ in \eqref{eq: Hankel penalized regression 1} to provide a sparse estimate for estimating the true order of the system, and the second time in \eqref{eq: Reduced order Hankel penalized regression 2} where it uses $2\check{d}_\xi + 1$ covariates to provide a more accurate estimator. To analyze the performance of this estimator, we first state a corollary to Theorem \ref{thm: stochatic control of the noise factor}.
\begin{corollary} \label{cor: control of noise term}
Under the same condition of Theorem \ref{thm: stochatic control of the noise factor}, for $\delta \in (0\ e^{-1})$, there exist an absolute positive constant $c > 0$ such that for $\lambda$ taken as
\begin{equation}\label{lambda}
    \lambda  := c\phi \sigma_u^2 \left( \sqrt{\frac{N_0}{N}} \vee \frac{\log(T)N_0}{N} \vee \frac{\sqrt{\log{\frac{1}{\delta}}}}{\sqrt{N }} \vee  \frac{\log(T)\log{\frac{1}{\delta}}}{N} \right),  
\end{equation}
with $\phi = |\bar{g}|_{\mathcal{H}_\infty} + \frac{\sigma_w}{\sigma_u}|h|_{\mathcal{H}_\infty} + 1$, we have, with probability at least $1-\delta$, the following upper bound:
\begin{equation}
       \lambda \geq  \frac{3}{\bar{N}}\left(|{H^{\dagger}}^{*}X^*\bar{X}\bar{g}|_{S_\infty} + |{H^{\dagger}}^{*}X^*Wh|_{S_\infty} + |{H^{\dagger}}^{*}X^*\varepsilon|_{S_\infty} \right).
\end{equation}
\end{corollary}

The following theorem provides various estimation bounds and the sample complexity for the Hankel penalized regression estimator of the Markov parameters for the $|\cdot|_{S_p}$-norms with $p \in (0\ 1)$. 
\begin{theorem} \label{thm: Performance of the Hankel penalized regression estimator}
Let $(X_{2T},y_{2T}), \dots, (X_N,y_N)$ be the input and output values of the LTI hidden state space model \eqref{LTI} under Gaussian or subGaussian assumption for the different noise vectors. Assume that $(X,y)$ is given by \eqref{eq: short hand matrix notation for LTI} and the estimator $\hat{g}$ given by \eqref{g-hat}. Define $\Delta g = \hat{g} - g_0$, where $g_0$ is the sequence of Markov parameters defined in \eqref{eq:IO-2 notation}. Under Assumptions \ref{Assumption: Stability} and \ref{Assumption: dimension less than its estimate}, for the values of $\bar{N}$ such that
\begin{equation} \label{eq: condition on sample size for Hankel penalized regression}
    \bar{N} \geq  c\left(TN_1 \vee T\log{\frac{1}{\delta}} \right)
\end{equation}
and the values of $\lambda$ given by \eqref{lambda},  with probability at least $1-2\delta$, for $\delta \in (0\ e^{-1}/2)$, the estimator $\hat{g}$ satisfies the following error bounds:
\begin{itemize}
    \item Slow and fast rates for the prediction error of the Markov parameters:
        \begin{equation} \label{eq: slow and fast rates prediction}
            \frac{1}{\sqrt{\bar{N}}} |X \Delta g^*|_{S_2} \leq \frac{5\sqrt{3}d_0^{1/2}T^{1/2}\lambda}{6\sigma_u} \wedge d_0^{1/4}\lambda^{1/2} |Hg_0^*|_{S_2}^{1/2}.
        \end{equation} 
    \item Slow and fast rates for the estimation error of the Markov parameters:
        \begin{equation} \label{eq: slow and fast rates estimation}
            |\Delta g|_{S_2}  \leq \frac{5\sqrt{3}d_0^{1/2}T^{1/2}\lambda}{6\sigma_u^2} \wedge \frac{\sqrt{2}d_0^{1/4}\lambda^{1/2}}{\sigma_u} |Hg_0^*|_{S_2}^{1/2}.
        \end{equation} 
    \item Fast rate for the Hankel estimation spectral loss:
        \begin{equation} \label{eq: Fast rate for the Hankel estimation}
            \mathcal{L}^H_2(\hat{g},g_0) \leq  \frac{5\sqrt{2}}{3\sigma_u^2}d_0^{1/2}\lambda T.
        \end{equation}
    \item Fast rate for the Hankel estimation $p$-loss, $p \in [1, 2]$:
        \begin{align} \label{eq: Fast rate for the Hankel estimation in Lp}
            \mathcal{L}^H_p(\hat{g},g_0) &\leq \frac{20 d_0^{1/p}\lambda T}{\sigma_u^2}.
        \end{align}
    \item Sample complexity for the spectral loss: for all $\epsilon > 0$ to obtain $\mathcal{L}^H_2(\hat{g},g_0) \leq \epsilon$ we need 
    \begin{align} \label{eq: sample complexity}
        \bar{N} \gtrsim \frac{\phi^2 d_0  T^2 N_0 }{\epsilon^2} \vee \frac{\phi d_0^{1/2}T N_0 \log(T)}{\epsilon} \vee \frac{\phi^2 d_0  T^2 \log{\frac{1}{\delta}}}{\epsilon^2} \vee \frac{\phi d_0^{1/2} T  \log(T) \log{\frac{1}{\delta}}}{\epsilon}.  
    \end{align}
\end{itemize}
\end{theorem}
\begin{remark}
Upon inspection of the proof we notice that the result would still hold without neither Assumptions \ref{Assumption: Stability} nor \ref{Assumption: dimension less than its estimate}. However, while all the rates hold without assumption $\ref{Assumption: Stability}$, this assumption is necessary for these rate to converge to $0$ when we observe more samples. Similarly, in the absence of Assumption \ref{Assumption: dimension less than its estimate} all the rates given in the theorem hold after replacing $d_0$ by $T$. Still, having $d_0 < T$ means that we are estimating less Markov parameter than necessary to be able to recover a minimal realization as was explained in Section \ref{sec: Problem statement and preliminaries}.
\end{remark}

The proof of Theorem \ref{thm: Performance of the Hankel penalized regression estimator} relies on the analysis of the first-order optimality condition. This approach appeared first in \cite{doi:10.1137/0329022} and in the case of matrix regression in \cite{10.1214/11-AOS894} to provide oracle inequalities in the context of low-rank matrix completion. The same argument can be combined with alternative approaches, including the analysis of the zero-order optimality condition suggested in \cite{10.1214/08-AOS620}. These are some of the approaches used for high dimension estimation problems. Indeed, we can cast the problem of estimating the Hankel matrix of a hidden state LTI state space model of unknown order as a high dimension matrix regression problem where we want to estimate a low rank Hankel matrix since the rank of the $T$ Hankel matrix of the Markov parameters is the dimension of the minimal realization $d_0$ as long as $T \geq d_0$. 

Existing results in the literature such as \cite[Theorem ~$3.1$]{9440770} are provided for the least square estimator in terms of the $|\cdot|_{S_\infty}$-norm while the dimension is known. They do not extend to the case of $|\cdot|_{S_p}$-norm with $p \in (0\ 1)$ since, while $|\cdot|_{S_\infty}$-norm is dimension free, the least square estimator is oblivious to the rank of the estimate. Indeed, the solution of the least square estimator is not expected to be low rank and thus by simple norm domination \cite[Theorem ~$3.1$]{9440770} implies,
\begin{equation*}
    |H\hat{g}^*_{\text{ls}} - Hg_0^*|_{S_2}\lesssim \sqrt{\frac{T^3q_0 \log^2(Tq_0)}{\bar{N}}}.
\end{equation*}   
This bound misses the correct dimension scaling by a polynomial factor of $T^{1/2}$ for the  $|\cdot|_{S_2}$-norm. On the other hand, Theorem 5.1 in \cite{JMLR:v22:19-725}  implies that
\begin{equation*}
    |\hat{\mathcal{H}}_d - Hg_0^*|_{S_2}\lesssim \sqrt{\frac{pT^3 +rT^2+T^2\log(1/\delta) }{\bar{N}}}
\end{equation*}
which also misses the correct dimension scaling by a factor of $(T/d_0)^{1/2}$. For a non-low rank estimator this is the expected order as it estimates $prT^2$ unknowns with a variance that scales like $T$ times the variance of $\bar{g}_0 \bar{X}_t + hW_t  + v_t$. We also note that the estimator $\hat{\mathcal{H}}_d$ does not preserve the Hankel structure of the matrix $Hg_0^*$. A low rank estimate reduces the number of the unknowns to $(p+r)d_0T$, which is consistent with our result which, after keeping only the main dimension terms, reads 
\begin{align*}
            \mathcal{L}^H_2(\hat{g},g_0) &\lesssim \sqrt{\frac{(p+r + \log(T/\delta) )d_0T^2}{N}} .
\end{align*}
In \cite[Theorem ~1]{pmlr-v120-sun20a} the authors study the problem of recovering the Markov parameter, while the dimension $d_0$ is unknown, but from the partial observation of multiple trajectories of the system and assuming that $w_i=0$ in \eqref{LTI}. To this end, they propose a penalized least square estimator for the Markov parameters as given in \eqref{eq: another penalized regressor}. While they successfully manage to control the error in the $|\cdot|_{S_\infty}$-norm, since they penalize with the transformation of the Hankel matrix $|HK^{-1}g^*|_{S_1}$, there is no reason to believe that the solution will give a low $\rank$ Hankel matrix. Their result \eqref{eq: for comparison lp bounds unknown dimension} in the $|\cdot|_{S_2}$-norm implies that with high probability and after observing enough data we have 
\begin{equation*}                                           \mathcal{L}^\mathcal{H}_2(\hat{g},g_0) \lesssim         \begin{cases}
        \frac{\sigma_z}{\sigma_u} \sqrt{\frac{rT^3}{N}} \log(T)& N \geq d_0^2\wedge T,\\
        \frac{\sigma_z}{\sigma_u} \sqrt{\frac{d_0rT^3}{N}} \log(T) & d_0 \leq N \leq d_0^2\wedge T,
        \end{cases}
    \end{equation*}
which again does not capture well the effect of the dimension. Moreover, in this case, it also misses the effect of the dynamic captured in our case by the term $\phi = |\bar{g}|_{\mathcal{H}_\infty} + \frac{\sigma_w}{\sigma_u}|h|_{\mathcal{H}_\infty} + 1$. This is due to the fact that in their setup, we stop every realization after $2T-1$ observation and suppose all trajectories are independent.
\begin{proof}
Set
\begin{equation*}
    \Gamma := |{H^{\dagger^*}} X^*\bar{X}\bar{g}^*|_{S_{\infty}} + |{H^{\dagger^*}} X^* Wh^* |_{S_{\infty}} + |{H^{\dagger^*}} X^* \varepsilon |_{S_{\infty}}.
\end{equation*} 
We start by using Corollary \ref{cor: control of noise term} and Theorem \ref{thm: Restricted eigenvalue property} to define an event of probability $1-2\delta$ where we have both
\begin{equation}
       \lambda \geq  \frac{3\Gamma}{\bar{N}} \label{eq: particular choice of lambda}
\end{equation}
and, for all $g \in \mathcal{M}_{p \times (2T-1)r}$,
\begin{equation} \label{eq: second inequality of the event}
    \frac{\sigma_u^2}{2}|g|_{S_2}^2 \leq \frac{1}{\bar{N}}|Xg^*|_{S_2}^2 \leq \frac{3\sigma_u^2}{2}|g|_{S_2}^2
\end{equation}
for values of $N$ such that
\begin{equation*}
    \bar{N} \geq  c\left(TN_1 \vee T\log{\frac{1}{\delta}} \right).
\end{equation*}
Since $\hat{g}$ solves the optimization problem \eqref{g-hat}, by Fermat's rule $0 \in \partial \Crit_{\lambda}(\hat{g})$, the subdifferential set of the criterion function. Also, by Fenchel-Rockafellar theorem (see e.g. \cite{peypouquet2015convex}), there exists $v \in \partial | H\hat{g}^*|_{S_1}$ such that, 
\begin{equation*}
    \frac{2}{\bar{N}} X^*(X\hat{g}^*-y)+\lambda v = 0.
\end{equation*} 
Using the fact that $ y = Xg^*_0 + \bar{X}\bar{g}_0^* + Wh_0^* + \varepsilon$ and multiplying by $\Delta g = \hat{g}^\lambda-g_0$ gives 
\begin{equation*}
    |X \Delta g^*|_{S_2}^2 = \frac{\lambda \bar{N}}{2}\innerl{-\Delta g^*}{v} + \innerl{X^*( \bar{X}\bar{g}_0^* + Wh_0^* + \varepsilon)}{\Delta g}.
\end{equation*} 
By the definition of the sub-gradient we have, for all $g \in \mathcal{M}_{p \times (2T-1)r}$,
\begin{equation*}
    | Hg^*|_{S_1} \geq | H\hat{g}^* |_{S_1} + \innerl{-\Delta g^*}{v}.
\end{equation*} 
H{\"o}lder's inequality yields
\begin{align}
    |X \Delta g^*|_{S_2}^2 &\leq   \innerl{X^* (\bar{X}\bar{g}^* + Wh^* + \varepsilon) }{\Delta g^*} + \frac{\lambda \bar{N}}{2}(| Hg^*_0|_{S_1} - | H\hat{g}^*|_{S_1})  \nonumber\\
    &\leq  \Gamma |H\Delta g^*|_{S_1} + \frac{\lambda \bar{N}}{2} (| Hg^*_0|_{S_1} - | H\hat{g}^*)|_{S_1}  \label{eq: main deterministic result}\\
    &\leq  (\Gamma + \frac{\lambda \bar{N}}{2})| Hg^*_0|_{S_1} +(\Gamma - \frac{\lambda \bar{N}}{2})| H\hat{g}^*)|_{S_1}. \nonumber  
\end{align}
Since by \eqref{eq: particular choice of lambda} we have $\lambda \geq  \frac{2 \Gamma}{\bar{N}}$, then it holds that
\begin{align*}
    \frac{1}{\bar{N}} |X\hat{g}^* - Xg_0^*|_{S_2}^2 \leq& \lambda |Hg_0^*|_{S_1} \leq \lambda \sqrt{\rank(Hg_0^*)}|Hg_0^*|_{S_2} 
\end{align*}
which proves the slow rate in \eqref{eq: slow and fast rates prediction}. The slow rate in \eqref{eq: slow and fast rates estimation} is implied by inequality \eqref{eq: slow and fast rates prediction}, since we are in an event where the inequality \eqref{eq: second inequality of the event} holds. 

For a matrix $M$ with a singular value decomposition $M= U\Sigma V^*$ define the projection operators $P_U^{\perp} := I - UU^*$, $P_V^{\perp} := I - V^*V$, $\Proj_{M}^{\perp}(N) := P_U^{\perp} N P_V^{\perp}$, and $\Proj_{M}^{\perp} := I - \Proj_{M}^{\perp}$. Since we have a decomposable penalty \cite{buhlmann2011statistics}, we have 
\begin{align*}
|H\hat{g}^*|_{S_1} &= |Hg_0^* + H\Delta g^*|_{S_1}, \\
&= |Hg_0^* + \Proj_{Hg_0^*}^{\perp}(H\Delta g^*) + \Proj_{Hg_0^*}(H\Delta g_0^*)|_{S_1}, \\
&\geq |Hg_0^* + \Proj_{Hg^*}^{\perp}(H\Delta g^*) |_{S_1} - | \Proj_{Hg_0^*}(H\Delta g^*)|_{S_1}, \\
&= |Hg_0^* |_{S_1} + | \Proj_{Hg_0^*}^{\perp}(H\Delta g^*) |_{S_1} - | \Proj_{Hg_0^*}(H\Delta g^*)|_{S_1},
\end{align*}
from which, together with \eqref{eq: main deterministic result}, we obtain
\begin{align*}
|X \Delta g^*|_{S_2}^2 &\leq  \Gamma |H\Delta g^*|_{S_1} + \frac{\lambda \bar{N}}{2}(| \Proj_{Hg_0}(H\Delta g^*)|_{S_1} - | \Proj_{Hg_0^*}^{\perp}(H\Delta g^*) |_{S_1} ) \\
&=  \Gamma(| \Proj_{Hg_0}(H\Delta g^*)|_{S_1} + | \Proj_{Hg_0^*}^{\perp}(H\Delta g^*) |_{S_1} ) + \frac{\lambda \bar{N}}{2} (| \Proj_{Hg_0^*}(H\Delta g^*)|_{S_1} \\
&\qquad \qquad- | \Proj_{Hg_0^*}^{\perp}(H\Delta g^*) |_{S_1} ) \\
&=  (\Gamma - \frac{\lambda \bar{N}}{2}) | \Proj_{Hg_0^*}^{\perp}(H\Delta g_*) |_{S_1} + ( \Gamma + \frac{\lambda \bar{N}}{2}) | \Proj_{Hg_0^*}(H\Delta g^*)|_{S_1}.
\end{align*}
Again, by the particular choice of $\lambda \geq \frac{3 \Gamma}{\bar{N}}$, we have 
\begin{equation}
0 \leq |X \Delta g^*|_{S_2}^2
\leq \frac{\lambda }{6} \bar{N} (5| \Proj_{Hg_0^*}(H\Delta g^*)|_{S_1}- |\Proj_{Hg_0^*}^{\perp}(H\Delta g^*) |_{S_1}). \label{eq: restricted isometry}
\end{equation}
Now, by the following rank inequality
\begin{align*}
\rank(\Proj_{Hg_0^*}(H\Delta g_0^*)) = \rank(\Proj_{Hg_0^*}(H\hat{g}^*) + Hg_0^*) \leq 2\rank(Hg_0^*).
\end{align*}
and the fact that we are on an event such that
\begin{equation*}
    \frac{\sigma_u^2}{2}|\Delta g^*|_{S_2}^2 \leq \frac{1}{\bar{N}}|X \Delta g^*|_{S_2}^2 \leq \frac{3\sigma_u^2}{2}|\Delta g^*|_{S_2}^2,
\end{equation*}
we have 
\begin{align*}
    &\frac{\sigma_u^2}{2T}|H\Delta g^*|_{S_2}^2 \leq \frac{\sigma_u^2}{2}|\Delta g^*|_{S_2}^2 \leq \frac{1}{\bar{N}}|X \Delta g^*|_{S_2}^2 \\
    &\leq \frac{5\lambda }{6}  \sqrt{\rank(\Proj_{Hg_0^*}(H\Delta g^*)}| \Proj_{Hg_0^*}(H\Delta g^*)|_{S_2}\\
    &\leq \frac{5\sqrt{2}\lambda }{6}  d_0^{1/2}| H\Delta g^*|_{S_2}\leq \frac{5\sqrt{2}\lambda }{6}  (d_0T)^{1/2}| \Delta g^*|_{S_2}  \\
    &\leq \frac{5\sqrt{3}\lambda }{6}  (d_0T)^{1/2}  \sqrt{\frac{1}{\bar{\bar{N}}}}\frac{|X\Delta g^*|_{S_2}}{\sigma_u}.
\end{align*}
This implies the fast rates  in \eqref{eq: slow and fast rates prediction}, \eqref{eq: slow and fast rates estimation}, and \eqref{eq: Fast rate for the Hankel estimation}.
Also, since $\hat{g}$ satisfies \eqref{eq: restricted isometry}, we have
\begin{equation} \label{eq: cone condition}
5| \Proj_{Hg_0^*}(H\Delta g^*)|_{S_1} \leq |\Proj_{Hg_0^*}^{\perp}(H\Delta g^*) |_{S_1}.
\end{equation}
This gives
\begin{align*}
    &\quad | H\Delta g^*|_{S_1} \leq |\Proj_{Hg_0^*}(H\Delta g^*)|_{S_1} + | \Proj_{Hg_0^*}^{\perp}(H\Delta g^*) |_{S_1} \\
    &\leq 6|\Proj_{Hg_0^*}(H\Delta g^*)|_{S_1}  
    \leq 6\sqrt{2} d_0^{1/2} |H\Delta g^*|_{S_2} \leq \frac{20 d_0^{1/2}\lambda T}{\sigma_u^2}.
\end{align*}
But, $|H\Delta g_0^*|_{S_p} = \left(\sum \limits_{i=1}^{d} s_{i}^p \right)^{1/p} = |s|_p$ where $s$ is the vector of singular values. Therefore, by the norm interpolation identity $|s|_p \leq (|s|_1)^{2/p-1}(|s|_2)^{2-2/p}$ for $p \in [0,1]$, we finally obtain
\begin{align*}
    |H\Delta g^*|_{S_p} &\leq (|H\Delta g^*|_{S_1})^{2/p-1}(|H\Delta g^*|_{S_2})^{2-2/p} \leq \frac{20 d_0^{1/p}\lambda T}{\sigma_u^2}.
\end{align*}
\end{proof}

\begin{remark}
The condition \eqref{eq: condition on sample size for Hankel penalized regression} on the sample size is likely to be sub-optimal. One expects that the factor $T$ should be replaced by $d_0$. The factor $T$ comes from the use of the concentration result of Theorem \ref{thm: Restricted eigenvalue property}. While this Theorem gives the right rate for the input covariates' concentration, the result is stronger than needed. Indeed, Theorem \ref{thm: Restricted eigenvalue property} provides us with an event in which for all $g \in \mathcal{M}_{p \times (2T-1)r}(\mathbb{R})$ we have

\begin{equation} 
    \frac{\sigma_u^2}{2}|g|_{S_2}^2 \leq \frac{1}{\bar{N}}|Xg^*|_{S_2}^2 \leq \frac{3\sigma_u^2}{2}|g|_{S_2}^2,
\end{equation} 
while the proof needs such a control only on the set defined by the cone condition \eqref{eq: cone condition}.
\end{remark}

\begin{open problem}
Show that for all $g \in \mathcal{M}_{p \times (2T-1)r}(\mathbb{R})$ such that $|g|_{S_2}\leq 1$ and \eqref{eq: cone condition} hold then 
\begin{equation*}
        \left|\frac{1}{\bar{N}}|Xg^*|_{S_2}^2-\sigma_u^2|g^*|_{S_2}^2\right| \lesssim \sigma_u^2\sqrt{\frac{d_0}{\bar{N}}}
        ,
\end{equation*}
up to logarithmic terms and lower order terms. 
\end{open problem}
While the condition \eqref{eq: condition on sample size for Hankel penalized regression} on $\bar{N}$ is likely to be suboptimal, we note that it is still less restrictive than the sample complexity \eqref{eq: sample complexity} which will play a major role in the analysis of Algorithm \ref{algo 1}. As we shall see below, for this reason, the condition \eqref{eq: condition on sample size for Hankel penalized regression}  will not affect the upcoming results on the estimation of the parameters $(\bar{A},\bar{B},\bar{C})$.

In the following proposition we show that the SVD decomposition of the Hankel matrix obtained from the Makov parameters estimate $\hat{g}$ given in \eqref{g-hat} can be used to recover the system's order $d_0$, if given a lower bound on the smallest singular value of the true Hankel matrix of Markov parameters $Hg^*_0$. We also show that the fast rate for the spectral loss in \eqref{eq: Fast rate for the Hankel estimation} implies a fast rate for the truncation of the SVD decomposition.

To this end, we consider the SVD decomposition of the Hankel matrix of the estimated parameter $\hat{g}$ given by $H\hat{g}^* = \sum \limits_{i=1}^{\rank(H\hat{g}^*)} \hat{s}_i \hat{u}_i \hat{v}_i^*$ and define the truncation dimension $\check{d}_{\xi}$ and the truncated SVD matrix $\hat{\mathcal{H}}_{\check{d}_{\xi}}$ of the estimated Hankel matrix as: 
\begin{equation}\label{d-xi}
    \check{d}_{\xi} := \sum \limits_{i=1}^{\rank(H\hat{g}^*
    )} \mathbf{1}\{\hat{s}_i \geq 2\xi \} \quad \text{and} \quad \hat{\mathcal{H}}_{\check{d}_{\xi}} := \sum \limits_{i=1}^{\rank(H\hat{g}^*
    )} \mathbf{1}\{\hat{s}_i \geq 2\xi \} \hat{u}_i \hat{v}_i^*.   
\end{equation}

\begin{proposition} \label{prop: Fast rate and rank recovery via SVD}
Assume the same conditions on $(X_{2T},y_{2T}), \dots, (X_N,y_N)$ as in Theorem \ref{thm: Performance of the Hankel penalized regression estimator} and suppose that 
$$
s_{d_0} (Hg_0^*) \geq 3\xi
$$ 
for some $\xi > 0$. Then, there exists an absolute positive constant $c$ such that for the values of $\bar{N}$ given by
\begin{equation} \label{eq: restatment of the condition on sample complexity}
    \bar{N} \geq c d_0 TN_0 \vee T\log{\frac{1}{\delta}} \vee \frac{\phi^2 d_0 T^2}{\xi^2} \left(N_0 \vee \log{\frac{1}{\delta}} \right) \vee \frac{\phi d_0^{1/2} T\log(T)}{\xi} \left(N_0 \vee \log{\frac{1}{\delta}} \right),
\end{equation}
the dimension $\check{d}_{\xi}$ and then estimate $\hat{\mathcal{H}}_{\check{d}_{\xi}}$ defined in \eqref{d-xi}, satisfy with probability at least $1-2\delta$ the following. 
\begin{itemize}
    \item Exact rank recovery:
        \begin{equation} \label{eq: SVD rank recovery}
            \check{d}_{\xi} = d_0.
        \end{equation} 
    \item Lower bound over the least singular value of the truncated estimate:
        \begin{equation} \label{eq: SVD minimum singular value control}
            s_{\check{d}_{\xi}}(H\hat{g}) \geq 2\xi.
        \end{equation}
    \item Lower bound over the singular values after the truncated threshold:
        \begin{equation} \label{eq: SVD tail control}
            \text{for all}\,\, d \in \llbracket \check{d}_{\xi}+1 , \rank(H \hat{g}^*)\rrbracket, \qquad \hat{s}_i \leq \left(\sum \limits_{i = d_0 + 1}^{\rank(H \hat{g}^*)}  \hat{s}_i^2\right)^{1/2}  \leq \xi.
        \end{equation}
    \item Fast rate for the truncated estimate on the $2$-loss:
        \begin{equation} \label{eq: SVD fast rate}
            |\hat{\mathcal{H}}_{\check{d}_{\xi}} -H g_0^*|_{S_2} \leq \frac{10\sqrt{2}}{3\sigma_u^2}d_0^{1/2}\lambda T.
        \end{equation}
\end{itemize}
\end{proposition}
\begin{proof}
By the obtained sample complexity \eqref{eq: sample complexity} it follows that the condition on $\bar{N}$ in \eqref{eq: restatment of the condition on sample complexity} implies that, for a large enough absolute constant $c$, 
\begin{equation*}
    \frac{5\sqrt{2}}{3\sigma_u^2}d_0^{1/2}\lambda T \leq \xi
\end{equation*}
on the same event defined in Theorem \ref{thm: Performance of the Hankel penalized regression estimator}. Therefore, in view the same Theorem and Weyl's inequality, with the same probability of at least $1-2\delta$, we also have
\begin{align*}
    |s_d(H\hat{g}^*)-s_d(Hg^*_0)|\leq |H\Delta g^*|_{S_\infty} \leq |H\Delta g^*|_{S_2} \leq \frac{5\sqrt{2}}{3\sigma_u^2}d_0^{1/2}\lambda T \leq \xi.
\end{align*}
Now, if we assume that $\rank(H\hat{g}^*) < d_0$, then $s_{\min}(Hg^*) = s_{d_0}(Hg^*_0)$ and $s_{d_0}(H\hat{g}^*) = 0$. Thus, again by Weyl's inequality, we have
\begin{align*}
    s_{\min}(Hg_0^*) &= |s_{d_0}(Hg_0^*) - s_{d_0}(H\hat{g}^*)| \leq |s_1(H\Delta g^*)| \\
    &\leq |H\Delta g^*|_{S_2} \leq \frac{5\sqrt{2}}{3\sigma_u^2}d_0^{1/2}\lambda T \leq \xi,
\end{align*}
which contradicts the assumption
\begin{equation*}
    s_{\min}(Hg) \geq 3\xi.
\end{equation*}
Therefore, $d_0 \geq \rank(H\hat{g})$. This also means that
\begin{equation} \label{eq: lower bound on the estimated singular value}
    |s_{d_0}(Hg) - s_{d_0}(H\hat{g})| \leq \xi \quad \text{and} \quad s_{d_0}(H\hat{g}) \geq 2\xi.
\end{equation}
Since we now know that $\hat{d} \geq d_0 $, we consider the following decomposition for the SVD representation 
\begin{equation*}
    \hat{\mathcal{H}} = \hat{\mathcal{H}}_{d_0} + \hat{\mathcal{H}}_{\bar{d}} = \sum \limits_{i=1}^{d_0} \hat{s}_i \hat{u}_i \hat{v}_i^* + \sum \limits_{i = d_0 + 1}^{\rank(H \hat{g})}  \hat{s}_i \hat{u}_i \hat{v}_i^*
    = \begin{bmatrix} U_d &
    U_{\bar{d}}\end{bmatrix} \begin{bmatrix} \Sigma_d &\\
    &\Sigma_{\bar{d}}\end{bmatrix}\begin{bmatrix} V_d^* \\
    V_{\bar{d}}^*\end{bmatrix}.
\end{equation*}
Now, as the truncated SVD decomposition to rank $d_0$ solves the optimization problem:
\begin{equation*}
    \hat{\mathcal{H}}_{\bar{d}} \in \arg \min \limits_{H:\ \rank(H)\leq d} |\hat{\mathcal{H}} - H|_{S_2},
\end{equation*}
we obtain
\begin{equation} \label{eq: truncated SVD solve best rank approximation}
    |\hat{\mathcal{H}}_{\bar{d}}|_{S_2} = \min \limits_{H:\ \rank(H)\leq d} |\hat{\mathcal{H}} - H|_{S_2}  \leq |H\Delta g^*|_{S_2}. 
\end{equation}
In particular, 
\begin{equation*}
     \text{for all}\,\,  d \in \llbracket d_0+1 , \rank(H \hat{g}^*)\rrbracket \qquad \hat{s}_i \leq \left(\sum \limits_{i = d_0 + 1}^{\rank(H \hat{g}^*)}  \hat{s}_i^2\right)^{1/2}  \leq \xi.
\end{equation*}
This inequality together with \eqref{eq: lower bound on the estimated singular value} yield the following  result on rank recovery:
\begin{equation*}
    \check{d}_{\xi} = \sum \limits_{i=1}^{\rank(H\hat{g}^*
    )} \mathbf{1}\{\hat{s}_i \geq 2\xi \} = d.
\end{equation*}
It also yields \eqref{eq: SVD tail control} as well. 

Now, since we have
\begin{align*}
    |\hat{\mathcal{H}}_d -H g_0^*|_{S_2} &=|H\hat{g}  -H g_0^* - \hat{\mathcal{H}}_{\bar{d}}|_{S_2}
    \leq |H\Delta g|_{S_2} + |\hat{\mathcal{H}}_{\bar{d}}|_{S_2} \\
    &\leq 2|H\Delta g^*|_{S_2} \leq \frac{10\sqrt{2}}{3\sigma_u^2}d_0^{1/2}\lambda T,
\end{align*}
in view of \eqref{eq: truncated SVD solve best rank approximation}, we also have the fast rate for the SVD estimate in \eqref{eq: SVD fast rate}.
\end{proof}

\subsection{Error control for the Ho-Kalman algorithm estimates} \label{sec: Ho-Kalman}

This section provides stability results in the Hilbert-Schmidt norm for a version of an estimation procedure based on a variant of the Ho-Kalman algorithm. 
The variant of the Ho-Kalman algorithm in question is the one that obtains a minimal balanced realization starting from the SVD decomposition of the Hankel matrix of $T$ Markov parameters, for $T \geq d_0+1$. Indeed, the Ho-Kalman algorithm computes, up to a similarity transform, the observability and controllability matrices are respectively
    \begin{equation} \label{eq: constructuion of the observabillity and controllabilty}
        \bar{\mathcal{O}} =U_0 \Sigma_0^{1/2} \quad \text{ and } \quad \bar{\mathcal{C}} =\Sigma_0^{1/2} V_0^*,    
    \end{equation} 
and the minimal balanced realization (see Definition \ref{definition minimal balanced realization}) defined by
    \begin{equation*}
        \bar{A} := \left(\bar{\mathcal{O}}_{1:r(T-1),1:d_0}\right)^{\dagger}\bar{\mathcal{O}}_{r+1:rT,1:d_0},\,\,\, \bar{B} := \bar{\mathcal{O}}_{1:d_0,1:r},\,\,\,  \bar{C} := \bar{\mathcal{O}}_{1:p,1:d_0}.
    \end{equation*}
    
    Assuming that we have obtained an estimate $\hat{\mathcal{H}}_T$ of the Hankel matrix of order $T$ with a $\rank$ that is higher than the dimension $d_0$ and of the dimension $\check{d}_{\xi}$ such that the Hilbert-Schmidt error $| \hat{\mathcal{H}}_{T} - Hg_0^* |_{S_2}$ is small. 

The Ho-Kalman based estimation algorithm we introduce here yields an estimate of the minimal balanced realization by mimicking the Ho-Kalman algorithm described above. It starts from a truncated SVD decomposition $\hat{\mathcal{H}}_{\check{d}_{\xi}} = \hat{U}_\xi \hat{\Sigma}_\xi \hat{V}_\xi^*$ of the matrix $\hat{\mathcal{H}}_{T}$ to the smaller estimated dimension $\check{d}_{\xi}$ and constructs estimates of both the observability and controllablity matrices  
\begin{equation*}
    \hat{\mathcal{O}}:=\hat{U}_\xi \hat{\Sigma}_\xi^{1/2}, \quad \hat{\mathcal{C}}:=\hat{\Sigma}_\xi^{1/2} \hat{V}_\xi^*.
\end{equation*}
Thereafter, it provides an estimated minimal balanced realization as
\begin{equation*}
    \hat{A} = \left(\hat{\mathcal{O}}_{1:r(T-1),1:d_\xi}\right)^{\dagger}\hat{\mathcal{O}}_{r+1:rT,1:d_\xi},\,\,\, \hat{B}:= \hat{\mathcal{O}}_{1:d_\xi,1:r},\,\,\, \hat{C}:= \hat{\mathcal{O}}_{1:p,1:d_\xi}.
\end{equation*}
The next theorem provides error bounds for these estimates under the assumption that we have $\check{d}_{\xi} = d_0$, 
\begin{theorem} \label{thm: stability of Ho-kalman}
Suppose that $\check{d}_{\xi} = d_0$. Set 
\begin{equation} \label{eq: notation for Ho-Kalman} 
    \bar{\mathcal{O}}_{1:r(T-1),1:d_0} = \bar{\mathcal{O}}^+.
\end{equation}
If the following stability assumption holds
\begin{equation} \label{eq: Assumption for Ho-Kalman}
       | \hat{\mathcal{H}}_{T} - Hg_0^* |_{S_\infty} \wedge | \hat{\mathcal{H}}_{T} - Hg_0^* |_{S_2} \leq   \frac{\left(\sqrt{2} - 1\right)^{1/2}s_{d_0}(\bar{\mathcal{O}}^+)s_{d_0}^{1/2}(Hg_0^*)}{2\sqrt{2}},   
\end{equation}
then, there exists an orthonormal matrix $R$ such that the following holds. 
\begin{itemize}
    \item The error on the observability and controllability matrices is controlled by the error on the truncation:
        \begin{align*}
            &|\hat{\mathcal{O}} - \bar{\mathcal{O}}R|_{S_2}^2 + |\hat{\mathcal{C}}_d  - R^*\bar{\mathcal{C}}|_{S_2}^2 \leq \frac{2}{\sqrt{2}-1} \frac{|\hat{\mathcal{H}}_{\check{d}_{\xi}} -H g_0^*|_{S_2}^2}{s_{d_0}(Hg_0^*)}.
        \end{align*} 
    \item The error on the $C$ and $B$ matrices is controlled by the error on the truncation:
        \begin{align*}
            |\hat{C}   - \bar{C}R|_{S_2}^2 + | \hat{B}  - R^*\bar{B}|_{S_2}^2 \leq \frac{2}{\sqrt{2}-1} \frac{|\hat{\mathcal{H}}_{\check{d}_{\xi}} -H g_0^*|_{S_2}^2}{s_{d_0}(Hg_0^*)}.  
        \end{align*} 
    \item The error on the  $A$ matrix is controlled by the error on the truncation:
        \begin{equation*}
            |\hat{A} - R^*\bar{A} R|_{S_2} \leq  \frac{2^{3/2}\left(1 + \left|\bar{A} \right|_{S_\infty} \right)}{\left(\sqrt{2} - 1\right)^{1/2}s_{d_0}(\bar{\mathcal{O}}^+)s^{1/2}_{d_0}(Hg_0^*)}  |\hat{\mathcal{H}}_{\check{d}_{\xi}} -H g_0^*|_{S_2}.
        \end{equation*} 
\end{itemize}
\end{theorem}

This result provides a robustness analysis of the variant of a Ho-Kalman algorithm based estimation procedure described at the start of this section. The result is described in term of the $|\cdot|_{S_2}$-norm and shows that, under the stability condition \eqref{eq: Assumption for Ho-Kalman}, it is possible to recover up to an orthonormal matrix $R$ the minimal balanced realization defined in \ref{definition minimal balanced realization} since we can bound the loss function $\mathcal{L}^{\mathcal{M}}_2(\hat{\mathcal{M}}, \bar{\mathcal{M}})$ in term of $|\hat{\mathcal{H}}_{\check{d}_{\xi}} -H g_0^*|_{S_2}$ as follows. 
        \begin{equation*}
            \mathcal{L}^{\mathcal{M}}_2(\hat{\mathcal{M}}, \bar{\mathcal{M}}) \leq  \frac{2^{3/2}\left(1 + \left|\bar{A} \right|_{S_\infty} \right)}{\left(\sqrt{2} - 1\right)^{1/2}s_{d_0}(\bar{\mathcal{O}}^+)s^{1/2}_{d_0}(Hg_0^*)}  |\hat{\mathcal{H}}_{\check{d}_{\xi}} -H g_0^*|_{S_2}.
        \end{equation*}
In the next section we will use a slightly weaker version of this result to provide $\mathcal{L}^{\mathcal{M}}_2$ guarantees for Algorithm \ref{algo 1}, namely we replace $s_{d_0}(\bar{\mathcal{O}}^+)$ with $s^{1/2}_{d_0}(Hg_0^*)$ both in the robustness condition \eqref{eq: Assumption for Ho-Kalman} and in the error control of various estimates. This can be done since as argued in the proof $s_{d_0}(\bar{\mathcal{O}}^+) \leq s^{1/2}_{d_0}(Hg_0^*)$ and it is done so to only assume the knowledge of a lower bound  $s^{1/2}_{d_0}(\bar{\mathcal{O}}^+)$. Otherwise we could work with the original statement by assuming the knowledge of a lower bound on $s_{d_0}(\bar{\mathcal{O}}^+)s^{1/2}_{d_0}(Hg_0^*)$. The condition \eqref{eq: Assumption for Ho-Kalman} is stated with $| \hat{\mathcal{H}}_{T} - Hg_0^* |_{S_\infty} \wedge | \hat{\mathcal{H}}_{T} - Hg_0^* |_{S_2}$ which is always equal to $| \hat{\mathcal{H}}_{T} - Hg_0^* |_{S_\infty}$, it is done this way simply since sometimes it is easier to have a control over $| \hat{\mathcal{H}}_{T} - Hg_0^* |_{S_2}$ as it is the case for the Hankel penalized regression estimator in Theorem \ref{thm: Performance of the Hankel penalized regression estimator} .         

The version of the Ho-Kalman based estimator studied here is the one studied in \cite[Theorem ~$4$]{tsiamis2019finite}. Their guarantees suggested that
\begin{equation*}
    \mathcal{L}^{\mathcal{M}}_2(\hat{\mathcal{M}}, \bar{\mathcal{M}}) \leq  c\frac{ d_0\left|Hg_0^* \right|^{1/2}_{S_\infty} |\hat{\mathcal{H}}_{\check{d}_{\xi}} - H g_0^*|_{S_2}}{s_{d_0}^2(\bar{\mathcal{O}}^+)s^{1/2}_{d_0}(Hg_0^*)}  .
\end{equation*}
Our result improves it by replacing the factor $s_{d_0}^{-2}(\bar{\mathcal{O}}^+)s^{-1/2}_{d_0}(Hg_0^*)$ with the smaller factor $s_{d_0}^{-1}(\bar{\mathcal{O}}^+)s^{-1/2}_{d_0}(Hg_0^*)$ in the regime $s_{d_0}(Hg_0^*) \to 0$ and $s_{d_0}(\bar{\mathcal{O}}^+)\to 0$ which was introduced in remark \ref{remark: regime for upper bound comparison} and removing the $d_0$ factor. Another estimator based on the Ho-Kalman algorithm is considered in \cite[Theorem ~$5.2$]{pmlr-v120-sun20a} where it is shown that 
\begin{equation*}
    \mathcal{L}^{\mathcal{M}}_2(\hat{\mathcal{M}}, \bar{\mathcal{M}}) \leq  c\frac{ d_0^{1/2}\left|Hg_0^* \right|_{S_\infty} |\hat{\mathcal{H}}_{\check{d}_{\xi}} - H g_0^*|_{S_\infty}}{s^2_{d_0}(Hg_0^*)}  .
\end{equation*}
Here, we improve the factor $\frac{d_0^{1/2}\left|Hg_0^* \right|_{S_\infty} }{s^{2}_{d_0}(Hg_0^*)} |H\hat{g}^* -H g_0^*|_{S_\infty}$ by $\frac{\left|\bar{A}\right|_{S_\infty}|\hat{\mathcal{H}}_{\check{d}_{\xi}} -H g_0^*|_{S_2} }{s_{d_0}(\bar{\mathcal{O}}^+)s^{1/2}_{d_0}(Hg_0^*)}$ since $s_{d_0}(\bar{\mathcal{O}}^+)$ and $s^{1/2}_{d_0}(Hg_0^*)$  are usually comparable, as we shall see later in \eqref{eq: singular value comparison} where we have that  $s^{1/2}_{d_0}(Hg_0^*)\ge s_{d_0}(\bar{\mathcal{O}}^+)\ge \frac{1}{\sqrt{2}}s^{1/2}_{d_0}(Hg_0^*)$.
\begin{proof}
We start by noting that
\begin{align*}
     | \hat{\mathcal{H}}_{\check{d}_{\xi}} - Hg_0^* |_{S_\infty} \leq | \hat{\mathcal{H}}_{T} - \hat{\mathcal{H}}_{\check{d}_{\xi}} |_{S_\infty} + | \hat{\mathcal{H}}_{T} - Hg_0^* |_{S_\infty}. 
\end{align*}
Since, by assumption, we have $\check{d}_{\xi} = d_0$ and the truncated SVD also minimizes the operator norm cost, we have
\begin{equation*}
    |\hat{\mathcal{H}}_{T} - \hat{\mathcal{H}}_{\check{d}_{\xi}} |_{S_\infty} = \min_{\rank(H) \leq \check{d}_{\xi}} |\hat{\mathcal{H}}_{T} - H|_{S_\infty} \leq | \hat{\mathcal{H}}_{T} - Hg_0^* |_{S_\infty}.
\end{equation*}
Since $\bar{\mathcal{O}}^+$ is a sub-matrix of $\bar{\mathcal{O}}$, we have
\begin{equation} \label{eq: estimates of the smallest singular value}
    s_{d_0}(\bar{\mathcal{O}}^+) \leq s_{d_0}(\bar{\mathcal{O}}) \leq  s_{d_0}^{1/2} (Hg_T^*),
\end{equation}
where the second inequality follows from the construction in \eqref{eq: constructuion of the observabillity and controllabilty}. Therefore,  the condition \eqref{eq: Assumption for Ho-Kalman} implies 
\begin{equation} \label{eq: implication of the condition}
    | \hat{\mathcal{H}}_{\check{d}_{\xi}} - Hg_0^* |_{S_\infty} \leq 2 | \hat{\mathcal{H}}_{T} - Hg_0^* |_{S_\infty} \leq 2 | \hat{\mathcal{H}}_{T} - Hg_0^* |_{S_2} \leq \frac{s_{d_0}(Hg_0^*)}{2}.
\end{equation}
The result for the error on both the observability and controllability matrices will be derived as a direct consequence of the following lemma taken from \cite{pmlr-v48-tu16}. A similar approach was used by \cite{pmlr-v120-sun20a} to analyze the performance of another variant of the Ho-Kalman algorithm. 
\begin{lemma}[Lemma 5.14 in \cite{pmlr-v48-tu16}]
Let $M_1,\ M_2 \in  \mathcal{M}_{n_1\times n_2}(\mathbb{R})$ be two $\rank$ $r$ matrices with SVD decompositions $M_1 = U_1\Sigma_1 V_1^*$ and
$M_2 = U_2 \Sigma_2 V_2^*$. If $  | M_2 - M_1 |_{S_\infty} \leq \frac{s_r(M_1)}{2}$ then there is an orthonormal matrix $R$ such that:
\begin{equation*}
    |U_1 \Sigma_1^{1/2} - U_2 \Sigma_2^{1/2} R|^2_{S_2} + |V_1 \Sigma_1^{1/2}  - R^* V_1 \Sigma_1^{1/2}|^2_{S_2} \leq \frac{2| M_2 - M_1 |^2_{S_2}}{(\sqrt{2}-1)s_r(M_1)}.
\end{equation*} 
\end{lemma}
Since both $\hat{\mathcal{H}}_{\check{d}_{\xi}} $ and $Hg_0^*$ are of $\rank$ $d_0$, we can use this lemma together with \eqref{eq: implication of the condition} to guarantee on the same event that there exist a matrix $R$ such that $RR^* = I_d$ and 
\begin{align*}
    |\hat{\mathcal{O}} - \bar{\mathcal{O}}R|_{S_2}^2 + |\hat{\mathcal{C}}  - R^*\bar{\mathcal{C}}|_{S_2}^2 &=|\hat{U}_{\xi} \hat{\Sigma}_{\xi}^{1/2} - U_0 \Sigma_0^{1/2} R|_{S_2}^2 + | \hat{V}_{\xi}^*\hat{\Sigma}_{\xi}^{1/2}  - R^*V_0^* \Sigma_0^{1/2}|_{S_2}^2 \\
    &\leq \frac{2}{\sqrt{2}-1} \frac{|\hat{\mathcal{H}}_{\check{d}_{\xi}} -H g_0^*|_{S_2}^2}{s_{d_0}(Hg_0^*)}.
\end{align*}
Since, $\bar{C}$ and $\bar{B}$ are submatrices of $\mathcal{O}$ and $\mathcal{C}$ respectively, the last inequality implies
\begin{align*}
    &|\hat{C}   - \bar{C}R|_{S_2}^2 + | \hat{B}  - R^*\bar{B}|_{S_2}^2 \leq \frac{2}{\sqrt{2}-1} \frac{|\hat{\mathcal{H}}_{\check{d}_{\xi}} -H g_0^*|_{S_2}^2}{s_{d_0}(Hg_0^*)}.
\end{align*}
To derive the estimation error bound for the matrix $\bar{A}$, we recall the following notation (introduced in \eqref{eq: notation for Ho-Kalman}), 
\begin{align*}
    \hat{\mathcal{O}}_{r+1:rT,1:d_0}:= \hat{\mathcal{O}}^-,\ \bar{\mathcal{O}}_{1:r(T-1),1:d_0}:= \bar{\mathcal{O}}^+,\ \bar{\mathcal{O}}_{r+1:rT,1:d_0}:= \bar{\mathcal{O}}^-.
\end{align*}
We note that
\begin{align}
    &|\hat{A} - R^*\bar{A} R|_{S_2}= \left|\left(\hat{\mathcal{O}}^+\right)^{\dagger}\hat{\mathcal{O}}^- - R^*\bar{A} R \right|_{S_2} \nonumber\\
    &= \left|\left(\hat{\mathcal{O}}^+\right)^{\dagger}\hat{\mathcal{O}}^- - \left(\hat{\mathcal{O}}^+\right)^{\dagger}\hat{\mathcal{O}}^+R^*\bar{A} R \right|_{S_2} \leq  \left|\left(\hat{\mathcal{O}}^+\right)^{\dagger}\right|_{S_\infty} \left|\hat{\mathcal{O}}^- - \hat{\mathcal{O}}^+R^*\bar{A} R \right|_{S_2} \nonumber \\
    &\leq \frac{1}{s_{d_0}(\hat{\mathcal{O}}^+)} \left( \left|\hat{\mathcal{O}}^- - \bar{\mathcal{O}}^+RR^*\bar{A} R \right|_{S_2} + \left|\bar{\mathcal{O}}^+RR^*\bar{A} R - \hat{\mathcal{O}}^+R^*\bar{A} R \right|_{S_2}\right) \nonumber \\
    &\leq \frac{1}{s_{d_0}(\hat{\mathcal{O}}^+)} \left( \left|\hat{\mathcal{O}}^- - \bar{\mathcal{O}}^- R \right|_{S_2} + \left|\bar{\mathcal{O}}^+R - \hat{\mathcal{O}}^+\right|_{S_2}\left|\bar{A} \right|_{S_\infty}\right) \nonumber \\
    &\leq \left(\frac{2}{\sqrt{2} - 1}\right)^{1/2} \frac{\left(1 + \left|\bar{A} \right|_{S_\infty} \right)}{s_{d_0}(\hat{\mathcal{O}}^+)}  \frac{|\hat{\mathcal{H}}_{\check{d}_{\xi}} -H g_0^*|_{S_2}}{s^{1/2}_{d_0}(Hg_0^*)},\label{eq: control on A with wrong least singular value}
\end{align}
where in the last inequality we used the fact that both $\hat{\mathcal{O}}^- - \bar{\mathcal{O}}^- R $ and $  \hat{\mathcal{O}}^+ -\bar{\mathcal{O}}^+R$ are submatrices of $ \hat{\mathcal{O}}- \bar{\mathcal{O}}R$. By Weyl's inequality we have
\begin{align*}
    &|s_{d_0}(\hat{\mathcal{O}}^+) - s_{d_0}(\bar{\mathcal{O}}^+)| = |s_{d_0}(\hat{\mathcal{O}}^+) - s_{d_0}(\bar{\mathcal{O}}^+R)| \leq |\hat{\mathcal{O}}^+ -\bar{\mathcal{O}}^+R|_{S_\infty} \\
    &\leq |\hat{\mathcal{O}} -\bar{\mathcal{O}}R|_{S_\infty} \leq |\hat{\mathcal{O}} -\bar{\mathcal{O}}R|_{S_2} \leq \left(\frac{2}{\sqrt{2} - 1}\right)^{1/2} \frac{|\hat{\mathcal{H}}_{\check{d}_{\xi}} -H g_0^*|_{S_2}}{s^{1/2}_{d_0}(Hg_0^*)}. 
\end{align*}
Again, noting that
\begin{align*}
     | \hat{\mathcal{H}}_{\check{d}_{\xi}} - Hg_0^* |_{S_2} \leq | \hat{\mathcal{H}}_{T} - \hat{\mathcal{H}}_{\check{d}_{\xi}} |_{S_2} + | \hat{\mathcal{H}}_{T} - Hg_0^* |_{S_2} 
\end{align*}
and since the truncated SVD minimizes the Hilbert-Schmidt norm cost, we obtain
\begin{equation*}
    |\hat{\mathcal{H}}_{T} - \hat{\mathcal{H}}_{\check{d}_{\xi}} |_{S_2} = \min_{\rank(H) \leq \check{d}_{\xi}} |\hat{\mathcal{H}}_{T} - H|_{S_2} \leq | \hat{\mathcal{H}}_{T} - Hg_0^* |_{S_2}.
\end{equation*}
Therefore,
\begin{equation}
    | \hat{\mathcal{H}}_{\check{d}_{\xi}} - Hg_0^* |_{S_2} \leq 2 | \hat{\mathcal{H}}_{T} - Hg_0^* |_{S_2} 
\end{equation}
and
\begin{align*}
    &|s_{d_0}(\hat{\mathcal{O}}^+) - s_{d_0}(\bar{\mathcal{O}}^+)| \leq \left(\frac{2}{\sqrt{2} - 1}\right)^{1/2} \frac{|\hat{\mathcal{H}}_{\check{d}_{\xi}} -H g_0^*|_{S_2}}{s^{1/2}_{d_0}(Hg_0^*)}. 
\end{align*}
In view of the condition
$\frac{| \hat{\mathcal{H}}_{T} - Hg_0^* |_{S_2}}{s_{d_0}^{1/2}(Hg_0^*)} \leq  \frac{\left(\sqrt{2} - 1\right)^{1/2}s_{d_0}(\bar{\mathcal{O}}^+)}{2\sqrt{2}} $ in \eqref{eq: Assumption for Ho-Kalman}, we have
\begin{align*}
    s_{d_0}(\hat{\mathcal{O}}^+) &\geq s_{d_0}(\bar{\mathcal{O}}^+) -  \left(\frac{2}{\sqrt{2} - 1}\right)^{1/2} \frac{|\hat{\mathcal{H}}_{\check{d}_{\xi}} -H g_0^*|_{S_2}}{s^{1/2}_{d_0}(Hg_0^*)}  \\
    &\geq s_{d_0}(\bar{\mathcal{O}}^+) - \frac{s_{d_0}(\bar{\mathcal{O}}^+)}{2}  = \frac{s_{d_0}(\bar{\mathcal{O}}^+)}{2},
\end{align*}
which together with \eqref{eq: control on A with wrong least singular value} yields
\begin{equation*}
    |\hat{A} - R^*\bar{A} R|_{S_2} \leq  \frac{2^{3/2}\left(1 + \left|\bar{A} \right|_{S_\infty} \right)}{\left(\sqrt{2} - 1\right)^{1/2}s_{d_0}(\bar{\mathcal{O}}^+)s^{1/2}_{d_0}(Hg_0^*)}  |\hat{\mathcal{H}}_{\check{d}_{\xi}} -H g_0^*|_{S_2}.
\end{equation*}
\end{proof}
\subsection{Non-asymptotic guarantees for Algorithm \ref{algo 1}} \label{sec: Non-assumptotic guarantees for algorithm}

Now we are ready to derive non-asymptotic results for the complete estimation procedure described in Algorithm \ref{algo 1}. The algorithm starts with the data obtained from the partial observation of a single trajectory $(X_i,y_i)_{i=2T}^N$ of the system and aims to obtain a possible realization $(\hat{A},\hat{B},\hat{C})$. To this end, we require the inputs $T_0$, $\lambda_0$, and $\xi$ to satisfy the following conditions:
\begin{enumerate}
    \item $T_0 \geq d_0 + 1$, a known strict upper bound for the system order which can be taken reasonably large at the expense of an additional cost in terms of the sample complexity \eqref{eq: sample complexity}, as it directly relates to the dimension of the unknowns in the Hankel penalized regression part of the algorithm.   
    \item $\lambda_0 \simeq \lambda$ as defined in \eqref{eq: Hankel penalized regression 1}. This choice requires the additional knowledge of an upper bound for $\phi \sigma_u^2$ as defined in Corollary \ref{cor: control of noise term}. An upper bound on $\phi$ is obtained from an upper bound on the system's $\mathcal{H}_\infty$-norm and an upper bound on the variances of the involved random variables. As argued in \cite{JMLR:v22:19-725}, the knowledge of an upper bound on the system $\mathcal{H}_\infty$-norm is a plausible assumption. It was also shown in \cite{8431846} that such upper bound could be efficiently estimated.
    \item $ s_{d_0}^2(\bar{\mathcal{O}}^+) \geq 5\xi$. This choice is made to establish a detection threshold. The assumption on the knowledge of such a threshold is also common in the literature when studying threshold based estimator for high dimension regression problems, see \cite[Corollary ~2]{10.1214/07-AOS582}, \cite[Equation ~(8)]{JMLR:v7:zhao06a} or \cite[Assumption  ~$3$]{10.1214/08-EJS177}.   
\end{enumerate}  
\begin{remark}
Since we want to provide an estimate up to a similarity transform of a minimal realization, as explained in Section \ref{sec: Problem statement and preliminaries}, the fact that a realisation is minimal is equivalent to the order $T$ Observability (\emph{resp} Controllability) matrix being full column (\emph{resp} row) rank. This implies that the two requirements, $T \geq d_0$ and $s_{d_0}(\mathcal{O})>0$ should be satisfied. Hence, the conditions $T_0 \geq d_0 + 1$ and $ s_{d_0}^2(\bar{\mathcal{O}}^+) \geq 5\xi$ strengthen those requirements to a level that permits the estimation and {\it rank} detection. 

The condition $\lambda_0 \simeq \lambda$ relates to how the system's dynamic affects the estimation error, through the variance term $\phi$ in \eqref{eq: Fast rate for the Hankel estimation}. Assuming the knowledge of an upper bound on it is again strengthening this requirement to a level that permits the estimation and $\rank$ detection. 

Obtaining an adaptive, entirely data-driven estimation procedure without those three additional inputs falls beyond the scope of the current paper and is left as an interesting extension for future work. 
\end{remark}

Denote $g_{0,T} = [CB,CAB,\cdots,CA^{T-1}B]$ so that $Hg_{0,T}^* \in \mathcal{M}_{rT \times pT}$. Since $s_{d_0}^{1/2}(Hg_{0,d_0+1}^*) \geq s_{d_0}(\bar{\mathcal{O}})\geq s_{d_0}(\bar{\mathcal{O}}^+)$, then our choice of $ s_{d_0}^2(\bar{\mathcal{O}}^+) \geq 5\xi$ also implies
\begin{equation} \label{eq: condition on singular values simplified}
    s_{d_0}(\bar{\mathcal{O}}^+)s_{d_0}^{1/2}(Hg_{0,T}^*) \geq 3\xi \quad \text{and} \quad 
     \frac{\left(\sqrt{2} - 1\right)^{1/2}s_{d_0}(\bar{\mathcal{O}}^+)s_{d_0}^{1/2}(Hg_0^*)}{2\sqrt{2}} \geq \xi. 
\end{equation}
According to Theorem \ref{thm: Performance of the Hankel penalized regression estimator} with the choices of inputs above and for 
\begin{equation*}
    \bar{N} \geq  c\left(TN_1 \vee T\log{\frac{1}{\delta}} \right)
\end{equation*}
when $\lambda_0$ is taken as 
\begin{equation*}
    \lambda_0= c \phi \sigma_u^2\left( \sqrt{\frac{N_0}{N}} \vee \frac{\log(T_0)N_0}{N} \vee \frac{\sqrt{\log{\frac{1}{\delta}}}}{\sqrt{N }} \vee  \frac{\log(T_0)\log{\frac{1}{\delta}}}{N} \right),  
\end{equation*}
the Hankel penalized estimator $\hat{g}$ defined in \eqref{eq: Hankel penalized regression 1} satisfies on an event $\mathcal{B}$ of probability $\mathbb{P}(\mathcal{B}) \geq 1-2\delta$ a fast rate for the Hankel estimation spectral loss  
\begin{align*}
    \mathcal{L}^H_2(\hat{g}_0,g_0) \leq  \frac{5\sqrt{2}}{3\sigma_u^2}d_0^{1/2}\lambda_0 T_0,
\end{align*}
for some absolute fixed positive constant $c$ and  $\phi$ as defined in Corollary \ref{cor: control of noise term}. By equation \eqref{eq: condition on singular values simplified} the above choice of $\xi$ implies $s_{d_0} (Hg_{0,T_0}^*) \geq 3\xi$, Thus Proposition \ref{prop: Fast rate and rank recovery via SVD} means that on the same event, once we have 
\begin{equation*}\begin{array}{lll}
    \bar{N} \geq \bar{N}_0 = c d_0 TN_0 \vee T_0\log{\frac{1}{\delta}} \vee \frac{\phi^2 d_0 T_0^2}{\xi^2} \left(N_0 \vee \log{\frac{1}{\delta}} \right) \\ \qquad\qquad\quad \vee \frac{\phi d_0^{1/2} T_0\log(T_0)}{\xi} \left(N_0 \vee \log{\frac{1}{\delta}} \right),
    \end{array}
\end{equation*}
we can ensure exact rank recovery for Algorithm \ref{algo 1} in the sense that $\check{d}_{\xi}$ defined by 
$
\check{d}_{\xi} = \sum \limits_{i=1}^{\rank(H\hat{g}^*)} \mathbf{1}\{\hat{s}_i \geq 2\xi \}  
$
satisfies $\check{d}_{\xi} = d_0$.
Hence, the event $\mathcal{B}$ is included in the event $\{\check{d}_{\xi} = d_0\}$.

In a similar fashion, using the Hankel penalized regression estimator in \eqref{eq: Reduced order Hankel penalized regression 2} with $T_1 = d_0 + 1$ to get an estimate for the Hankel matrix of the Markov parameters, then on a event $\mathcal{A}$ of probability $\mathbb{P}(\mathcal{A}) \geq 1- 2\delta$ and for 
\begin{equation*}
    \bar{N} \geq c (d_0 + 1)N_1 \vee (d_0 + 1)\log{\frac{1}{\delta}}, 
\end{equation*}  
we have a fast rate for the new Hankel estimation spectral loss:  
\begin{align*}
    \mathcal{L}^H_2(\hat{g}_1,g_0) &\leq \frac{5\sqrt{2}}{3\sigma_u^2}d_0^{1/2}\lambda_1 (d_0 + 1)
\end{align*}
with 
\begin{equation*}
    \lambda_1  = c\phi \left( \sqrt{\frac{N_0}{N}} \vee \frac{\log(d_0 + 1)N_0}{N} \vee \frac{\sqrt{\log{\frac{1}{\delta}}}}{\sqrt{N }} \vee  \frac{\log(d_0 + 1)\log{\frac{1}{\delta}}}{N} \right). 
\end{equation*}
This estimated matrix is then used in the Ho-Kalman based estimation procedure to obtain estimates for the system parameters $(\hat{A},\hat{B},\hat{C})$. As long as $\check{d}_{\xi} = d_0$, Theorem \ref{thm: stability of Ho-kalman} guarantees that for values of $\bar{N}$ such that
\begin{equation} \label{eq: second condition for the sample size}
     | \hat{\mathcal{H}}_{d_0+1} - Hg_{0,d_0+1}^* |_{S_2} \leq  \frac{\left(\sqrt{2} - 1\right)^{1/2}s_{d_0}(\bar{\mathcal{O}}^+)s_{d_0}^{1/2}(Hg_0^*)}{2\sqrt{2}}, 
\end{equation}
there exists an orthonormal matrix $R$ satisfying
\begin{itemize}
    \item up to the same orthonormal transformation, a fast estimation rate for $\bar{C}$ and $\bar{B}$ given as
        \begin{align*}
            |\hat{C}   - \bar{C}R|_{S_2}^2 + | \hat{B}  - R^*\bar{B}|_{S_2}^2 \leq \frac{2^{1/2}}{(\sqrt{2}-1)^{1/2}} \frac{|\hat{\mathcal{H}}_{\check{d}_{\xi}} -H g_{0,d_0+1}^*|_{S_2}}{s_{d_0}(\bar{\mathcal{O}}^+)},
        \end{align*} 
    \item a fast estimation rate for $\bar{A}$ given as
        \begin{equation*}
            |\hat{A} - R^*\bar{A} R|_{S_2} \leq  \frac{2^{3/2}\left(1 + \left|\bar{A} \right|_{S_\infty} \right)}{\left(\sqrt{2} - 1\right)^{1/2}s_{d_0}^2(\bar{\mathcal{O}}^+)}  |\hat{\mathcal{H}}_{\check{d}_{\xi}} -H g_{0,d_0+1}^*|_{S_2}.
        \end{equation*}
\end{itemize}
From Proposition \ref{prop: Fast rate and rank recovery via SVD} equation \eqref{eq: SVD fast rate} we have the following fast rate
\begin{equation*}
    |\hat{\mathcal{H}}_{\check{d}_{\xi}} -H g_{0,d_0+1}^*|_{S_2} \leq \frac{10\sqrt{2}}{3\sigma_u^2}(d_0 + 1)^{3/2}\lambda_1.
\end{equation*}
From \eqref{eq: condition on singular values simplified}, the condition \eqref{eq: second condition for the sample size} is satisfied as long as $| \hat{\mathcal{H}}_{d_0+1} - Hg_{0,d_0+1}^* |_{S_2} \leq \xi$. In view of the sample complexity given in \eqref{eq: sample complexity} in Theorem \ref{thm: Performance of the Hankel penalized regression estimator}, this is the case if 
\begin{equation*}
    \bar{N} \geq \bar{N}_1  = c\frac{\phi^2 d_0^3}{\xi^2} \left(N_0 \vee \log{\frac{1}{\delta}} \right) \vee \frac{\phi d_0^{3/2} \log(d_0)}{\xi} \left(N_0 \vee \log{\frac{1}{\delta}} \right).
\end{equation*}

For Algorithm \ref{algo 1} to succeed we should have both the events $\{T_1 = d_0 + 1\}$ and $\mathcal{A}$ occurring. For $\bar{N} \geq \bar{N}_0 \vee \bar{N}_1 = \bar{N}_0$ we obtain
\begin{align*}
    \mathbb{P}(\{T_1 = d_0 + 1\} \cap \mathcal{A}) &\geq \mathbb{P}(\{T_1 = \check{d}_{\xi} + 1\} \cap \{\check{d}_{\xi} = d_0\} \cap \mathcal{A})\\
    &= \mathbb{P}( \{\check{d}_{\xi} = d_0\} \cap \mathcal{A}) \geq \mathbb{P}( \mathcal{A} \cap \mathcal{B})\\
    &\geq 1 - 4\delta,
\end{align*}
where in the equality we used the fact that Algorithm \ref{algo 1} always chooses $\{T_1 = \check{d}_{\xi} + 1\}$. In the second inequality we use the fact that our choice  $\bar{N} \geq \bar{N}_0$ ensures $\mathcal{B} \subset \{\check{d}_{\xi} = d_0\}$, and in the third inequality we use the fact that both $\mathbb{P}(\mathcal{A}) \geq 1-2\delta$ and $\mathbb{P}(\mathcal{B}) \geq 1-2\delta$ and a union bound.

We summarize the results of this discussion in the following
\begin{theorem} \label{thm: performance of the algorithm}
Algorithm \ref{algo 1} succeeds with probability at least $1-4\delta$ for all $\delta \in (0\ e^{-1}/4)$ after observing $\bar{N} \geq \bar{N}_0 $ samples from a single trajectory of the system \eqref{LTI} with the particular choices $T_0$, $\lambda_0$, $\lambda_1$, and $\xi$ as described above and $T_1 = \check{d}_{\xi} + 1$. On the event of success we have
\begin{itemize}
    \item Exact order recovery $\check{d}_{\xi} = d_0$;
    \item There exist an orthonormal matrix $R$ for which the estimates for $C$ and $B$ satisfies fast estimation rates given as
        \begin{align*}
            |\hat{C}   - \bar{C}R|_{S_2} + | \hat{B}  - R^*\bar{B}|_{S_2} \leq \frac{20(d_0 +1)^{3/2}\lambda_1}{3(\sqrt{2}-1)^{1/2}s_{d_0}(\bar{\mathcal{O}}^+)\sigma_u^2}. 
        \end{align*} 
    \item For the same matrix $R$ the estimate for $A$ also satisfies fast estimation rate given as:
        \begin{equation*}
            |\hat{A} - R^*\bar{A} R|_{S_2} \leq  \frac{10 \left(1 + |\bar{A}|_{S_\infty} \right)(d_0 +1)^{3/2}\lambda_1}{\left(\sqrt{2} - 1\right)^{1/2}s^2_{d_0}(\bar{\mathcal{O}}^+)\sigma_u^2}.
        \end{equation*} 
\end{itemize}
\end{theorem}
In particular, we have the following 
\begin{corollary}
Under the same condition as Theorem \ref{thm: performance of the algorithm}, the same inputs for Algorithm \ref{algo 1}and for $\bar{N}\geq \bar{N}_0$,  with the same probability on the even of success, the output satisfies the following.
\begin{itemize}
        \item Fast rate for the Hankel estimation spectral loss:
        \begin{equation*}
            \mathcal{L}^{\mathcal{M}}_2(\hat{\mathcal{M}},\bar{\mathcal{M}}) \leq  \frac{10 \left(1 + |\bar{A}|_{S_\infty} \right)(d_0 +1)^{3/2}\lambda_1}{\left(\sqrt{2} - 1\right)^{1/2}s^2_{d_0}(\bar{\mathcal{O}}^+)\sigma_u^2}.
        \end{equation*}
    \item Sample complexity for the spectral loss: for any $\epsilon > 0$, to obtain $\mathcal{L}^H_2(\hat{g},g_0) \leq \epsilon$ we need 
    \begin{multline} \label{eq: sample complexity for algorithm} 
        \bar{N} \gtrsim \frac{\phi^2 (1 + |\bar{A}|_{S_\infty} )^2d_0^3 N_0 }{s^4_{d_0}(\bar{\mathcal{O}}^+)\epsilon^2} \vee \frac{(1 + |\bar{A}|_{S_\infty} )d_0^{3/2} \phi N_0 \log(T)}{s^2_{d_0}(\bar{\mathcal{O}}^+)\epsilon} \\
        \vee \frac{\phi^2 (1 + |\bar{A}|_{S_\infty} )^2d_0^3 \log{\frac{1}{\delta}}}{s^4_{d_0}(\bar{\mathcal{O}}^+)\epsilon^2} \vee \frac{(1 + |\bar{A}|_{S_\infty} )
        d_0^{3/2} \phi \log(T) \log{\frac{1}{\delta}}}{s^2_{d_0}(\bar{\mathcal{O}}^+)\epsilon}.
    \end{multline}
\end{itemize}
\end{corollary}
It is clear from the condition \eqref{eq: Assumption for Ho-Kalman} in Theorem \ref{thm: stability of Ho-kalman} that the condition $\bar{N} \geq \bar{N}_0$ could be improved by using the control in term of $| \hat{\mathcal{H}}_{T} - Hg_0^* |_{S_\infty}$ instead of $| \hat{\mathcal{H}}_{T} - Hg_0^* |_{S_2}$. This can be done using the least square estimator for which it is easier to derive estimation bounds in term of $| \hat{\mathcal{H}}_{T} - Hg_0^* |_{S_\infty}$ such as in \cite[Theorem ~$3.1$]{9440770}, which would have the effect of reducing $\bar{N}_0$ by a factor of $d_0$ so that it scales like $T_0^2$. In the high dimension regime, $T_0^2$ is still a big price to pay in comparison with the sample complexity necessary for the second stage. This suggest the following open problem where we only change the condition on $\bar{N}$ in Theorem \ref{thm: Performance of the Hankel penalized regression estimator},
\begin{open problem}
Is there an algorithm that successfully learns a minimal realization of a Hidden state LTI state space system with high probability after observing 
\begin{equation*}
    \bar{N} \gtrsim \frac{\phi^2 d_0^2}{\xi^2} \left(N_0 \vee \log{\frac{1}{\delta}} \right)
\end{equation*} 
and satisfies the fast rate for the Hankel estimation spectral loss of Theorem \ref{thm: performance of the algorithm} given as
        \begin{equation*}
            \mathcal{L}^{\mathcal{M}}_2(\hat{\mathcal{M}},\bar{\mathcal{M}}) \leq  \frac{\left(1 + |\bar{A}|_{S_\infty} \right) d_0^{3/2}\lambda_1}{s^2_{d_0}(\bar{\mathcal{O}}^+)\sigma_u^2}
        \end{equation*}
up to logarithmic terms and lower order terms. 
\end{open problem}

In \cite[Theorem ~$5.3$]{JMLR:v22:19-725} the authors study the parameter estimation problem of a hidden state LTI state space system of unknown order where the derived results are in $|\cdot|_{\infty}$ norm. We have summarized their results in the Related Literature section in the introduction. The dominant term for the error bound \eqref{eq: for comparison parametric bounds} for their algorithm, after multiplying by $\hat{d}^{1/2}$ to get a bound in the  $|\cdot|_{2}$ norm, is 
\begin{equation*} 
    \mathcal{L}^{\mathcal{M}}_2(\hat{\mathcal{M}},\bar{\mathcal{M}}) \lesssim \sqrt{\frac{r \hat{d}^4 + p \hat{d}^5 + \hat{d}^4 \log(N/\delta)}{s_{\hat{d}}(\hat{H})N}} .
\end{equation*}
Our result improves this in a few ways. First, our result is provided in terms of the actual dimension $d_0$ and not the estimated dimension $\hat{d}$. Moreover, it reduces this dependence by a factor of $d_0$.

Regarding the burn in time $\bar{N}_0$, we have an explicit number of samples required for the result to hold $\bar{N} \geq \bar{N}_0$, unlike in \cite[Theorem ~$5.3$]{JMLR:v22:19-725} where the result hold for $\bar{N} \geq N_*$ with $N_* < \infty$ \eqref{eq: for comparison parametric bounds}. Moreover, upon inspection, a combination of \cite[Proposition ~$13.4$ and Proposition ~$13.7$]{JMLR:v22:19-725} shows that $N_*$ depends exponentially in $\hat{d}$.

Finally, the presence of $\Gamma (\hat{H},\varepsilon) < \infty$ in \eqref{eq: for comparison parametric bounds} makes the result sensitive to all the singular values gapes of the matrix $\hat{H}$ and not just on the location of the smallest one. Our result, on the other hand, does not exhibit such behavior.

For the 'Reduced order Hankel penalized regression' part of Algorithm \ref{algo 1}, let us consider a value $T_1 = \bar{d}_0 + 1$ where we take
\begin{equation} \label{eq: choice to fix singular value issue}
    \bar{d}_0 = \check{d}_{\xi} \vee \check{\eta} \quad \text{with} \quad \check{\eta} \geq \eta = \frac{\log\left(\psi_{\bar{A}}|\bar{C}|_{S_\infty}^2/s_{d_0}(Hg_{0,d_0})\right)}{2\log\left( 1/\rho(\bar{A}) \right)}, 
\end{equation}
which on the even of success becomes $\bar{d}_0 = d_0 \vee \check{\eta}$. Applying \ref{thm: stability of Ho-kalman} we obtain  the same guarantees of \ref{thm: performance of the algorithm} except that we replace $d_0$ with $\bar{d}$. In this case, 
    \begin{equation*}
        \bar{\mathcal{O}}^+ = \begin{bmatrix}
        \bar{C}\\
        \vdots\\
        \bar{C}\bar{A}^{\bar{d} - 1}\\
        \end{bmatrix}
    \end{equation*}
    which is of rank $d_0$, since $k_0 \geq 1$. Moreover,
    \begin{align*}
        s^2_{d_0}(\bar{\mathcal{O}}^+) &= \inf \limits_{|u|_2=1} \left\{|\mathcal{O}u|_{S_2}^2 - |\bar{C}\bar{A}^{\bar{d}}|_{S_2}^2 \right\}\geq s^2_{d_0}(\mathcal{O}) - |\bar{C}|_{S_\infty}^2|\bar{A}^{\bar{d}}|_{S_\infty}^2\\
        &=  s_{d_0}(Hg_{0,\bar{d}}) - |\bar{C}|_{S_\infty}^2|\bar{A}^{\bar{d}}|_{S_\infty}^2 \geq  s_{d_0}(Hg_{0,d_0}) - \psi_{\bar{A}} |\bar{C}|_{S_\infty}^2\rho(\bar{A})^{2\bar{d}}.
    \end{align*}
In view of the choice made in \eqref{eq: choice to fix singular value issue}, we have
\begin{align} \label{eq: singular value comparison}
        s^2_{d_0}(\bar{\mathcal{O}}^+) \geq   \frac{1}{2}s_{d_0}(Hg_{0,d_0}).
\end{align}
Thus, we have the following 
\begin{corollary}
Under the same condition as Theorem \ref{thm: performance of the algorithm}, for the same inputs, except for the additional $\check{\eta}$, by taking $T_1 = \bar{d}_0 + 1$ and $s_{d_0}(Hg_{0,d_0}) \geq 10\xi$ with $\bar{d}_0$ defined in \eqref{eq: choice to fix singular value issue}, Algorithm \ref{algo 1} succeeds for $\bar{N}\geq \bar{N}_0 \vee \bar{N}_{\eta}$ with
\begin{equation*}
    \bar{N}_{\eta} \geq  c\frac{\phi^2 \eta^3}{\xi^2} \left(N_0 \vee \log{\frac{1}{\delta}} \right) \vee \frac{\phi \eta^{3/2} \log(\eta)}{\xi} \left(N_0 \vee \log{\frac{1}{\delta}} \right).
\end{equation*}
Furthermore, the output $\hat{\mathcal{M}}$ satisfies with probability at least $1 - 4 \delta$ the following.
\begin{itemize}
        \item Fast rate for the Hankel estimation spectral loss:
        \begin{equation*}
            \mathcal{L}^{\mathcal{M}}_2(\hat{\mathcal{M}},\bar{\mathcal{M}}) \leq  \frac{20 \left(1 + |\bar{A}|_{S_\infty} \right)(\bar{d}_0 + \eta + 1)^{3/2}\lambda_1}{\left(\sqrt{2} - 1\right)^{1/2}s_{d_0}(Hg_{0,d_0})\sigma_u^2}.
        \end{equation*}
    \item Sample complexity for the spectral loss: for any $\epsilon > 0$, to obtain $\mathcal{L}^H_2(\hat{g},g_0) \leq \epsilon$, we need 
    \begin{multline} 
        \bar{N} \gtrsim \frac{\phi^2 (1 + |\bar{A}|_{S_\infty} )^2(d_0 + \eta )^3 N_0 }{s_{d_0}^2(Hg_{0,d_0})\epsilon^2} \vee \frac{(1 + |\bar{A}|_{S_\infty} )(d_0 + \eta )^{3/2} \phi N_0 \log(T)}{s_{d_0}(Hg_{0,d_0})\epsilon} \\
        \vee \frac{\phi^2 (1 + |\bar{A}|_{S_\infty} )^2(d_0 + \eta )^3 \log{\frac{1}{\delta}}}{s_{d_0}^2(Hg_{0,d_0})\epsilon^2}  \\ \vee \frac{(1 + |\bar{A}|_{S_\infty} )
        (d_0 + \eta )^{3/2} \phi \log(T) \log{\frac{1}{\delta}}}{s_{d_0}(Hg_{0,d_0})\epsilon}.
    \end{multline}
\end{itemize}
\end{corollary}

\bibliographystyle{imsart-number} 
\bibliography{sample}       

\newpage
\appendix

\section{Proofs of the main probabilistic results} \label{Appendix: Proofs of the main probabilistic results}
In this appendix we gather the proofs of the main results stated in Section \ref{sec: section on Probabilistic results} namely the proofs of Theorem \ref{thm: Restricted eigenvalue property} and of \eqref{eq: first noise term} of Theorem \ref{thm: stochatic control of the noise factor}. The proofs of the other parts of Theorem \ref{thm: stochatic control of the noise factor} are similar to those of \eqref{eq: first noise term} and are given in  Appendix \ref{Appendix: Remaining probabilistic results} for completeness. Their proofs use extensively generic chaining estimates. Thus, we start by recalling few concepts from the generic chaining literature  to fix some notation and refer to \cite{talagrand2006generic,gine_nickl_2015} for more on the topic. Let $(\mathcal{A},d)$ be a metric space. The distance of a point $t \in \mathcal{A}$ to a subset $\mathbb{A} \subseteq \mathcal{A}$ is defined as
\begin{equation*}
    d(t,\mathbb{A}) = \inf \limits_{s \in \mathbb{A} } d(t,s).
\end{equation*}
The diameter of the set $\mathbb{A}$ is 
\begin{equation*}
    \Delta(\mathbb{A}) = \sup \limits_{(s,t) \in \mathbb{A}^2} d(t,s),
\end{equation*}
and the covering number $N(\mathcal{A},d,u)$ is the smallest number of balls in $(\mathcal{A},d)$ of radius less than $u$  needed to cover $\mathcal{A}$ (\emph{i.e.}, whose union includes $\mathcal{A}$). A ball of center $c \in \mathcal{A}$ and radius $r\ge 0$ with respect to a distance $d$ or a metric $\lVert\cdot\rVert$ will be denoted $B_d(c,r)$ or $B_{\lVert\cdot\rVert}(c,r)$, respectively.

The gamma-$\alpha$ functional $\gamma_\alpha(\mathcal{A},d)$ for the metric space $(\mathcal{A}, d)$ and its corresponding upper bound by the Dudley chaining integral are defined as follows. 
\begin{align}
    \gamma_\alpha(\mathcal{A},d) := \inf \sup \limits_{t \in \mathcal{A}} \sum \limits_{r= 0}^{\infty} 2^{r/\alpha} d(t,\mathbb{A}_r) 
    \lesssim \int_{0}^{ \Delta(\mathcal{A})} (\log N(\mathcal{A}, d,u))^{1/\alpha}du,\label{ekv}
\end{align}
where the infimum is taken over all sequences of sets $(\mathbb{A}_r)_{r \in \mathbb{N}}$ in $\mathcal{A}$ with $|\mathbb{A}_0|=1$ and $|\mathbb{A}_r| \leq 2^{2^r}$ (\cite{talagrand2006generic}). If $d(x,y) = \lVert x-y\rVert$ for some norm $\lVert\cdot\rVert$ as it is usually the case, we also use the notation $\gamma_\alpha(\mathcal{A},\lVert\cdot\rVert)$ for $\gamma_\alpha(\mathcal{A},d)$.
\subsection{Isometric Property for the covariates of the input}

\begin{proof}[Proof of Theorem \ref{thm: Restricted eigenvalue property}]  
The result is an extension of \cite[Thoerem~$3.4$]{10.3150/20-BEJ1262} to the multidimensional case, and in the same spirit, we start the proof with a decomposition of the operator norm $|X^* X - \E (X^* X) |_{S_\infty}$ into the sum of 3 terms. To that end, we start by defining, for $k \in \llbracket1,2T-1\rrbracket$, the following shifted matrices:
\begin{align*}
L_k = \begin{bmatrix}
0 & \cdots & 0 & u_{2T-1} & u_{2T} &\cdots &u_{\bar{N}} &0 &\cdots &0
\end{bmatrix}^*,    
\end{align*}
where $x_{2T-1}$ is at the position $r(k - 1)+ 1$. Then we define the matrices 
\begin{equation*}
    L = [L_1 L_2 \dots L_{2T-1}] \qquad \textrm{and} \qquad S = X - L,
\end{equation*}
to get a decomposition $X = L + S$ where $L$ and $S$ are independent of each other and have a shifted diagonal structure. Thus, we have
\begin{equation*}
    X^* X = L^* L + S^* S + S^* L + L^* S.
\end{equation*}
Using this decomposition the operator norm of deviation of $X^* X$ is upper bounded by
\begin{align}
     &|X^* X - \E (X^* X) |_{S_\infty} \leq |L^* L - \sigma_u^2(N-4T +3)I_{(2T-1)r}|_{S_\infty}  \nonumber\\
    &\qquad \qquad \qquad + \sigma_u^2(2T-2)+ 2|S^* L|_{S_\infty} +  |S^* S|_{S_\infty}. \label{eq: 3 terms to upper-bound}
\end{align}

We thus need to derive high probability bounds for the last three terms. Below, we give the derivation for the first term. The others are treated similarly, and the contribution of the first term dominates their contribution.  

We start by relating the operator norm of $L^* L - \sigma_u^2(N-4T +3)I_{(2T-1)r}$ to the supremum of a multiplication process. Since the columns of $L$ are shifted versions of each others, we have 
\begin{align*}
L^* L=\begin{bmatrix}
L_1^* L_1 &L_1^*L_2 & &L_1^*L_{2T-1} \\
L_2^*L_1 &L_1^*L_1  & &L_1^*L_{2T-2} \\
& & &\\
L_{2T-1}^*L_1 &L_{2T-2}^*L_1 & &L_1^*L_1
\end{bmatrix}. 
\end{align*}
Define the block Toeplitz operator $\mathcal{T}:\, l_{\mathbb{R}^r}(\mathbb{Z}) \to l_{\mathbb{R}^r}(\mathbb{Z})$ by the infinite diagonals of block matrices given
\begin{equation*}
    \mathcal{T}_0 = L_1^* L_1 - \sigma_u^2(N-4T +3)I_r,\ \mathcal{T}_l =L_1^*L_{l+1},\ \text{and} \ \mathcal{T}_{-l} =\mathcal{T}_l^*\ \text{for} \ l \in \llbracket 1, 2T-2\rrbracket.
\end{equation*}
The corresponding multiplication polynomial defined for $x \in [0,1]$ is given by 
\begin{equation*}
    p(x) = \sum \limits_{l=-2T+2}^{2T-2} \mathcal{T}_l e^{2 i \pi l x}.
\end{equation*}
Since $L^* L - \sigma_u^2(N-4T +3)I_{(2T-1)r}$ is a submatrix of $\mathcal{T}$, we have
\begin{equation}
    |L^* L - \sigma_u^2(N-4T +3)I_{(2T-1)r}|_{S_\infty} \leq |\mathcal{T}|_{2 \to 2} = \sup \limits_{x \in [0,1]} |p(x)|_{S_\infty},\label{eq : main HDP2}
\end{equation}
where $|\cdot|_{2 \to 2}$ stands for the operator norm.
The last supremum can also be expressed as 
\begin{multline*}
       \sup \limits_{x \in [0\ 1]} \sup \limits_{\substack{|v|_2 = 1\\
    |w|_2 = 1}}\Big| \sum \limits_{j=2T-1}^{\bar{N}} (\inner{u_j}{v}\inner{u_j}{w}- \sigma_u^2\inner{v}{w})\\
    + \sum \limits_{l=1}^{2T-2} \sum \limits_{j=2T-1+l}^{\bar{N}} \inner{u_j}{v}\inner{u_{j - l}}{w}  e^{2 i \pi l x}\\
    + \sum \limits_{l=1}^{2T-2} \sum \limits_{j=2T-1+l}^{\bar{N}} \inner{u_{j - l}}{v}\inner{u_j}{w}  e^{-2 i \pi l x}\Big|. 
\end{multline*}
Consider the block Toeplitz matrix $\mathcal{H} \in \mathcal{M}_{(N-4T +3)r \times (N -4T +3)r}(\mathbb{R})$ with block constant diagonals made of matrices $\mathcal{H}_l \in \mathcal{M}_{r \times r}(\mathbb{R})$ with $(j,k)$ entries 
\begin{equation*}
    \mathcal{H}_l(x,v,w) =  e^{2 i \pi l x} vw^* \mathds{1}\{ l \in \llbracket 0,2T-2\rrbracket\}\  \ \text{and} \ \mathcal{H}_l(x,v,w) =\mathcal{H}_{-l}(x,v,w)^* \ \text{for} \ l < 0.
\end{equation*}
Taking $u = [u_{2T-1}^*, \cdots ,u_{N-2T-1}^*]^*$ we obtain 
\begin{equation*}\begin{array}{ll}
    |L^* L - \sigma_u^2(N -4T +3)I_{(2T-1)r}|_{S_\infty}  \leq \sup \limits_{x \in [0,1]} |p(x)|_{S_\infty}  \\ \qquad\qquad\qquad\qquad\qquad\qquad\qquad\qquad= \sup \limits_{x \in [0\ 1]} \sup \limits_{\substack{|v|_2 = 1\\
    |w|_2 = 1}} |\inner{u}{\mathcal{H}(x,v,w)u}|.
    \end{array}
\end{equation*}
This defines a second order chaos process $\xi_{x,v,w} = \inner{u}{\mathcal{H}(x,v,w)u}$. We control its deviation using the Hanson-Wright inequality \cite{rudelson2013} to get, for all $t > 0$ with probability at least $1 - 2 e^{-ct}$, 
\begin{equation*}\begin{array}{ll}
     |\chi_{x_1,v_1,w_1}-\chi_{x_2,v_2,w_2}|\leq  \sqrt{t} d_2 ((x_1,v_1,w_1),(x_2,v_2,w_2)) \\ \qquad\qquad\qquad\qquad\qquad\quad + t d_\infty  ((x_1,v_1,w_1),(x_2,v_2,w_2)),
     \end{array}
\end{equation*}
where
\begin{align*}
    &d_\infty ((x_1,v_1,w_1),(x_2,v_2,w_2)) := |\mathcal{H}(x_1,v_1,w_1)-\mathcal{H}(x_2,v_2,w_2)|_{S_\infty}
\end{align*}
and
\begin{align*}
    &d_2 ((x_1,v_1,w_1),(x_2,v_2,w_2)) := |\mathcal{H}(x_1,v_1,w_1)-\mathcal{H}(x_2,v_2,w_2)|_{S_2}.
\end{align*}
The generic chaining result in ~\cite[Theorem ~$2.2.23$]{talagrand2006generic} and ~\cite[Theorem ~$3.5$]{dirksen2015} provides us with the following bound for the supremum of such mixed tail process for $t \geq 1$:
\begin{equation}
    \mathbb{P}\left( \sup \limits_{x \in [0,1]} |p(x)| \geq c  \sigma_u^2 \left( E +\sqrt{t} V + t U \right)\right) \leq 2 \exp{(-u)},\label{eq: HDP0 on polynomial}
\end{equation}
where 
\begin{align*}
    E &= \gamma_2([0,1] \times \mathbb{S}_2^{r-1} \times \mathbb{S}_2^{r-1},d_2) + \gamma_1([0,1] \times \mathbb{S}_2^{r-1} \times \mathbb{S}_2^{r-1}, d_\infty), \\
    V &= \Delta_2([0,1] \times \mathbb{S}_2^{r-1} \times \mathbb{S}_2^{r-1},d_2),\quad
    U= \Delta_\infty([0,1] \times \mathbb{S}_2^{r-1} \times \mathbb{S}_2^{r-1},d_\infty).
\end{align*}
To conclude the proof, it suffices to estimate these three terms. We start with few inequalities to simplify the involved the norm distances\begin{align*}
    &d_\infty ((x_1,v_1,w_1),(x_2,v_2,w_2)) := |\mathcal{H}(x_1,v_1,w_1)-\mathcal{H}(x_2,v_2,w_2)|_{S_\infty} \\
    &\leq 2\sup \limits_{y \in [0\ 1]}\left|\sum \limits_{l= 1}^{2T-2} e^{2 i \pi l (x_1 + y)}u_1v_1^* - e^{2 i \pi l (x_2 + y)}u_2v_2^*\right|_{S_\infty}+ \left| u_1v_1^* - u_2v_2^*\right|_{S_\infty}\\
    &\leq 2\sup \limits_{y \in [0\ 1]}\left|\sum \limits_{l= 1}^{2T-2} e^{2 i \pi l (x_1 + y)} - e^{2 i \pi l (x_2 + y)}\right|+ (4T-3)\left| u_1v_1^* - u_2v_2^*\right|_{S_\infty}\\
    &\lesssim T^2\left|x_1 -x_2\right|+ T\left| u_1 - u_2\right|_2+ T\left| v_1 - v_2\right|_2,
\end{align*}
where we used Proposition \ref{prop: control of nuclear norm by l2 norm} and the Liptchitz property of the complex exponential in the last step.  Similarly, we have
\begin{align*}
    &d_2 ((x_1,v_1,w_1),(x_2,v_2,w_2)) := |\mathcal{H}(x_1,v_1,w_1)-\mathcal{H}(x_2,v_2,w_2)|_{S_2} \\
    &\leq \sqrt{N-4T + 3}\left(2 \left( \sum \limits_{l= 1}^{2T-2}\left| e^{2 i \pi l x_1}u_1v_1^* - e^{2 i \pi l x_2}u_2v_2^*\right|_{S_2}^2\right)^{1/2}+ \left| u_1v_1^* - u_2v_2^*\right|_{S_2}\right)\\
    &\leq \sqrt{N-4T + 3}\left(2\left|\sum \limits_{l= 1}^{2T-2} (e^{2 i \pi l x_1} - e^{2 i \pi l x_2})^2\right|^{1/2}+ (2\sqrt{T-1}+1)\left| u_1v_1^* - u_2v_2^*\right|_{S_2}\right)\\
    &\lesssim \sqrt{N-4T + 3}\left(T^{3/2}\left|x_1 -x_2\right|+ T^{1/2}\left| u_1 - u_2\right|_2+ T^{1/2}\left| v_1 - v_2\right|_2\right).
\end{align*}
The radii $U$ and $V$ become
\begin{equation}
    U \lesssim T \ \text{and} \ V \lesssim \sqrt{( N-4T + 3) T}. \label{eq : HDP1}
\end{equation}
The $\gamma_1$ functional is evaluated as 
\begin{align}
    & \gamma_1([0,1] \times \mathbb{S}_2^{r-1} \times \mathbb{S}_2^{r-1}, d_\infty) \lesssim  \gamma_1([0,1] , T^2\left|\cdot \right|) +\gamma_1( \mathbb{S}_2^{r-1}, T\left| \cdot \right|_2)\\
    &\int_{0}^{T} \ln N([0,1], T^2 |\cdot|,u)du + \int_{0}^{T} \ln N(\mathbb{S}_2^{r-1}, T |\cdot|_2,u)du\\
    &= \int_{0}^{T} \ln \frac{T^2}{u} du + \int_{0}^{T} \ln \left(\frac{T}{u}\right)^r du\simeq  T (\ln(T) + r). \label{eq : HDP2}
\end{align}
Similarly, we can evaluate the $\gamma_2$ functional to get
\begin{align}
    \gamma_2([0,1] \times \mathbb{S}_2^{r-1} \times \mathbb{S}_2^{r-1},d_2)
    &\lesssim  \sqrt{T (N -4T + 3)(\ln(T) + r)}, \label{eq: HDP4}
\end{align}
\noindent Putting this last result together with \eqref{eq : HDP1}, \eqref{eq : HDP2}, \eqref{eq: HDP4} and \eqref{eq: HDP0 on polynomial} gives with probability at least $1-e^{-t}$ for $t \geq 1$:
\begin{multline*}
     |L^* L - \sigma_u^2(N -4T +3)I_{(2T-1)r}|_{S_\infty} \lesssim \sigma_u^2 \Big(T (\ln (T) + r) \\+ \sqrt{T (N -4T + 3)(\ln(T) + r)} + \sqrt{T (N -4T + 3)}\sqrt{t} +  Tt \Big).
\end{multline*}
We can bound the second and third term in \eqref{eq: 3 terms to upper-bound} by modifying the argument in \cite[Thoerem~$3.4$]{10.3150/20-BEJ1262} the same way we did here for the first term. This give us with probability at least $1- e^{-t}$, for all $t \geq 1$,
\begin{equation}
    |S^* S |_{S_\infty} \lesssim \sigma_u^2 \Big(T (\ln (T) + r )+ Tt\Big)\label{eq : main HDP1}
\end{equation}
and
\begin{equation*}
       |L^* S|_{S_\infty} \lesssim \sigma_u^2 \Big( T(\ln(T) + r)+  Tt \Big).
\end{equation*}
A straightforward union bound implies that with probability at least $1-e^{-t}$, for all $t \geq 1$, we have
\begin{align*}
    &|X^* X - \E (X^* X) |_{S_\infty} \lesssim \sigma_u^2 \Big( T (\ln(T) + r)+ \sqrt{T (N-4T + 3) (\ln(T)+r)}\\ 
    &\qquad \qquad \qquad \qquad\qquad + \sqrt{(N-4T + 3) T}t^{1/2} + Tt \big)\sigma_u^2 \Big).
\end{align*}
This last expression directly gives the claimed result in the theorem after normalization.
\end{proof}

\begin{proof} [Proof of \eqref{eq: first noise term} in Theorem \ref{thm: stochatic control of the noise factor}]
Define the permuted index 
\begin{equation*}
    \bar{l} = \begin{cases}
    l-2T \quad \text{  if } T+1\leq l\leq 2T-1,\\
l  \qquad \qquad \textrm{otherwise}.
\end{cases}
\end{equation*} 
and
\begin{align*}
    x &= [u^*_0,u^*_1,\dots,u^*_{N-2},u^*_{N-1}]^* \in \mathbb{R}^{Nr}\\    
    U_l &= [u_{2T-l},u_{2T+1-l},\dots,u_{N-l}]  \in \mathcal{M}_{r  \times \bar{N} }(\mathbb{R}).
\end{align*}
In view of the definition of $H^{\dagger^*}$ in \eqref{Adjoint} we have
\begin{equation*}
{H^{\dagger^*}}X^*\bar{X}\bar{g}^* = \begin{bmatrix}
&(U_1\bar{X}\bar{g}^*)^*  &\frac{1}{2} (U_2\bar{X}\bar{g}^*)^* &\dots &\frac{1}{T} (U_T\bar{X}\bar{g}^*)^* \\
& \frac{1}{2} (U_2\bar{X}\bar{g}^*)^* & \frac{1}{3}(U_3\bar{X}\bar{g}^*)^* &\dots &\frac{1}{T-1} (U_{T+1}\bar{X}\bar{g}^*)^* \\
&\vdots &\vdots &\vdots &\vdots\\
&\frac{1}{T} (U_T\bar{X}\bar{g}^*)^* &\frac{1}{T-1} (U_{T+1}\bar{X}\bar{g}^*)^* &\dots &  (U_{2T-1}\bar{X}\bar{g}^*)^*
\end{bmatrix}.
\end{equation*}
So we want to find a high probability bound on the operator norm of the matrix $H^{\dagger^*}X^*\bar{X}\bar{g}$. 
Define the infinite block Hankel operator $\mathcal{H}: \ l_2(\mathbb{N}) \to l_2(\mathbb{N})$ by the $\mathcal{M}_{p \times r }(\mathbb{R})$ blocks
\begin{align*}
\mathcal{H}_{i,j} = \begin{cases}
 1/|\bar{l}|(U_l\bar{X}\bar{g})^* \quad \text{for} \quad (i,j) \in \mathbb{N}^2, \ \text{and} \ 1\leq|i-j|=l \leq 2T-1,\\
0 \quad \text{otherwise}.
\end{cases} 
\end{align*}
Then
\begin{equation*}
    |H^{\dagger^*}X^*\bar{X}\bar{g}^*|_{S_\infty} \leq |\mathcal{H}|_{2 \to 2}
\end{equation*}
where $|\cdot|_{2 \to 2}$ stands for the operator norm from $l_2(\mathbb{Z})$ to $l_2(\mathbb{Z})$. The corresponding multiplication polynomial defined for $u \in [0,1]$ is given by 
\begin{equation*}
    p(u) = \sum \limits_{l=1}^{2T-1} \frac{1}{|\bar{l}|} U_l\bar{X}\bar{g}^* \exp{(i2 \pi \bar{l} u)}.
\end{equation*}
where we have used a permuted Fourier basis by the mapping $l \to \bar{l}$. Thus, using Proposition \ref{prop: upper bound on the operator norm of the multiplication process}, we obtain
\begin{align*}
|H^{\dagger^*}X^*\bar{X}\bar{g}^*|_{S_\infty} &\leq \sup \limits_{u \in [0\ 1]} |p(u)|_{S_\infty}= \sup \limits_{u \in [0,\ 1]} \left|\sum \limits_{l=1}^{2T-1} \frac{e^{i2 \pi \bar{l} u}}{|\bar{l}|} U_l\bar{X}\bar{g}^* \right|_{S_\infty} \\
&= \sup \limits_{u \in [0\ 1]} \sup \limits_{\substack{|v|_2 = 1\\
    |w|_2 = 1}} \left| \innerl{\sum \limits_{l=1}^{2T-1} \frac{e^{i2 \pi \bar{l} u}}{|\bar{l}|} U_l\bar{X}\bar{g}^*}{vw^*} \right| \\
    &= \sup \limits_{u \in [0\ 1]} \sup \limits_{\substack{|v|_2 = 1\\
    |w|_2 = 1}} \left| \innerl{\bar{X}\bar{g}^*}{\sum \limits_{l=1}^{2T-1} \frac{e^{i2 \pi \bar{l} u}}{|\bar{l}|} U_l^*vw^*} \right|.
\end{align*}
Define, for $l \in [0\ N-2T]$, the vectors $G_l \in \mathcal{M}_{p \times rN}(\mathbb{R})$ as
\begin{align*}
        G_l &= [C_0A_0^{2T-1+l}B_0,C_0A_0^{2T+l}B_0,\dots,C_0A_0^{2T-1}B_0,0,\dots,0],
\end{align*}
and for $u \in [0\ 1]$ the matrix valued functions $w_{u,v,w,l} \in \mathcal{M}_{r \times p}(\mathbb{C})$ by $w_{u,v,w,l} = \frac{\exp(i2\pi \bar{l}u)}{|\bar{l}|}vw^*$ for $l \in [1\ 2T-1]$ and the matrix valued functions $W_{u,v,w,l} \in \mathcal{M}_{p \times Nr}(\mathbb{C})$ by
\begin{align*}
W_{u,v,w,k} = \begin{bmatrix}
 & & 0 & w_{u,v,w,2T-1}^* & w_{u,v,w,2T-2}^* &\cdots &w_{u,v,w,1}^* & &
\end{bmatrix} ,    
\end{align*}
with the 1st zero a the $k^{\text{th}}$-position. Define $G$ as $G = [G_0^*,\dots,G_{N-2T}^*]^* $ and $W_{u,v,w}$ as $W_{u,v,w} = [W_{u,v,w,1}^*,\dots,W_{u,v,w,N-2T-1}^*]^*$ satisfying 
\begin{equation*}
    \innerl{\bar{X}\bar{g}^*}{\sum \limits_{l=1}^{2T-1} \frac{e^{i2 \pi \bar{l} u}}{|\bar{l}|} U_l^*vw^*}  = \innerl{W_{u,v,w}x}{Gx}.
\end{equation*}
This gives
\begin{align*}
|{H^{\dagger}}^{*}X^*\bar{X}\bar{g}|_{S_\infty} &\leq \sup \limits_{u \in [0\ 1]} \sup \limits_{\substack{|v|_2 = 1\\
    |w|_2 = 1}} |\innerl{G^*W_{u,v,w}x}{x}|
\end{align*}
which is the supremum of a second order chaos process defined by
\begin{align*}
\chi_{u,v,w} =  \innerl{G^*W_{u,v,w}x}{x}.
\end{align*}
To control the increment of the process we use Hanson Wright  inequality \cite{rudelson2013} which yields
\begin{equation*}
    \mathbb{P}\left( |\chi_{u,v,w} - \mathbb{E}(\chi_{u,v,w})|\geq c \left(\sqrt{t} |G^*W_{u,v,w}|_{S_2} + t |G^*W_{u,v,w}|_{S_\infty}  \right) \right) \leq 2 \exp{(-ct)}.
\end{equation*}
Since $ \mathbb{E}(\chi_{u,v,w}) = 0$, we obtain a mixed tail process with probability $1-e^{-t}$
\begin{equation}\begin{array}{lll}
    |\chi_{u_1,v_1,w_1} - \chi_{u_2,v_2,w_2}| \leq  c  \Big(\sqrt{t} |G^*\left(W_{u_1,v_1,w_1}-W_{u_2,v_2,w_2})\right)|_{S_2} \\
    \qquad\qquad\qquad\qquad\qquad\qquad + t |G^*\left(W_{u_1,v_1,w_1}-W_{u_2,v_2,w_2})\right)|_{S_\infty}  \Big) \\
   \qquad\qquad\qquad\qquad\qquad \leq  c |G|_{S_\infty} \left(\sqrt{t} |W_{u_1,v_1,w_1}-W_{u_2,v_2,w_2}|_{S_2} \right. \\ \left.
     \qquad\qquad\qquad\qquad\qquad\qquad + t |W_{u_1,v_1,w_1}-W_{u_2,v_2,w_2}|_{S_\infty}  \right).
    \end{array}
\end{equation}
Consider the pseudo-distances $d_2$ and $d_{\infty}$ defined on $[0,1]\times \mathbb{S}_2^{r-1}\times \mathbb{S}_2^{p-1}$ by  
\begin{multline*}
    \qquad \qquad d_2((u_{1},v_1,w_1),(u_2,v_2,w_2)) = |W_{u_1,v_1,w_1}-W_{u_2,v_2,w_2}|_{S_2} \\
    d_\infty ((u_{1},v_1,w_1),(u_2,v_2,w_2)) = |W_{u_1,v_1,w_1}-W_{u_2,v_2,w_2}|_{S_\infty}.\qquad \qquad
\end{multline*} 
The generic chaining result proved independently in ~\cite[Theorem ~$2.2.23$]{talagrand2006generic} and ~\cite[Theorem ~$3.5$]{dirksen2015} provides the following bound for the supremum of such mixed tail process for $t \geq 1$:
\begin{equation}
    \mathbb{P}\left( \underset{(u,v,w) \in [0,1] \times \mathbb{S}_2^{r-1 \times \mathbb{S}_2^{p-1}}}{\sup} |\chi_{u,v,w}| \geq C  \sigma_u^2 |G|_{S_\infty} \left( E +\sqrt{t} \Delta_{S_2} + t \Delta_{S_\infty} \right)\right) \leq 2 \exp{(-t)}, \label{eq: main estimate for noise level}
\end{equation}

where 
\begin{align*}
    &E = \gamma_2([0,1] \times \mathbb{S}_2^{r-1} \times \mathbb{S}_2^{p-1} ,d_2) + \gamma_1([0,1] \times \mathbb{S}_2^{r-1} \times \mathbb{S}_2^{p-1},d_\infty), \\
    &\Delta_{S_2} = \sup \limits_{(u_{1},v_1,w_1),(u_2,v_2,w_2) \in [0,1] \times \mathbb{S}_2^{r-1} \times \mathbb{S}_2^{p-1} }d_2((u_{1},v_1,w_1),(u_2,v_2,w_2)),\\
    &\Delta_{S_\infty} = \sup \limits_{(u_{1},v_1,w_1),(u_2,v_2,w_2) \in  [0,1] \times \mathbb{S}_2^{r-1} \times \mathbb{S}_2^{p-1} }d_\infty((u_{1},v_1,w_1),(u_2,v_2,w_2)).
\end{align*}
To conclude the proof, it suffices to estimate these four terms. We start with the terms which involves the distance $d_\infty$. An estimate of $d_\infty$ is obtained by seeing $W_{u,v,w}$ as a sub-matrix of an infinite Toeplitz matrix $\Tilde{W}_{u,v,w}$ defined by,
\begin{equation*}
    \Tilde{W}_{u,v,w} = \left[\mathds{1}_{1 \leq j-k+1\leq 2T-1} (w_{u,v,w}^*)_{j-k+1}\right]_{(j,k)\in \mathbb{Z}^2}
\end{equation*}
and the corresponding multiplication polynomial is
\begin{align*}
    q_{u,v,w}(z) &= \sum \limits_{l=1}^{2T-1} \frac{\exp{(i2 \pi \bar{l} (u+z))}}{|\bar{l}|}vw^*.
\end{align*}
The diameter $\Delta_{S_\infty}$ becomes
\begin{align*}
    \Delta_{S_\infty} &\leq 2 \sup \limits_{u \in [0\ 1]}\sup \limits_{\substack{|v|_2 = 1\\
    |w|_2 = 1}} |\Tilde{W}_{u,v,w}|_{2 \to 2} = 2 \sup \limits_{\substack{ u \in [0\ 1]\\
    z \in [0\ 1]}} \sup \limits_{\substack{|v|_2 = 1\\
    |w|_2 = 1}} |q(z)|_{S_2}\\
    &= 2 \sup \limits_{\substack{ u \in [0\ 1]\\
    z \in [0\ 1]}} \left|\sum \limits_{l=1}^{2T-1} \frac{\exp{(i2 \pi \bar{l} (u+z))}}{|\bar{l}|} \right| \leq \sum \limits_{l=1}^{T-1} \frac{2}{|\bar{l}|} \lesssim \log(T).
\end{align*}
Using the fact that the complex exponential is Lipschitz, we have 
\begin{align*}
    d_\infty((u_{1},v_1,w_2),(u_{1},v_1,w_2)) &\leq \sum \limits_{l=1}^{T-1} \frac{1}{|\bar{l}|}|v_1w_1^* - v_2w_2^* |_{S_2} \\ & \qquad + \left| \sum \limits_{l=1}^{T-1} \frac{1}{|\bar{l}|}(e^{i2 \pi l (u_1-z)} - e^{i2 \pi l (u_2-z)}) \right|\\
    &\lesssim \log(T) |v_1 - v_2 |_2+ \log(T) |w_1 - w_2 |_2 + T |u_1-u_2|.
\end{align*}
The $\gamma_1$ functional is evaluated as 
\begin{align*}
    &\gamma_1([0,1]\times \mathbb{S}_2^{r-1} \times \mathbb{S}_2^{p-1},d_{S_\infty}) \leq \gamma_1([0,1], T |\cdot|) + \gamma_1(\mathbb{S}_2^{r-1},\log(T) |\cdot |_2) \\ & \qquad\qquad\qquad\qquad\qquad\qquad\qquad + \gamma_1(\mathbb{S}_2^{p-1},\log(T) |\cdot |_2)\\
    &\lesssim \int_{0}^{ \Delta_{S_\infty}  } \log N([0,1], T |\cdot|,u)du + \int_{0}^{ \Delta_{S_\infty}  } \log N(\mathbb{S}_2^{r-1},\log(T) |\cdot|_2,u)du  \\ & \qquad + \int_{0}^{ \Delta_{S_\infty}  } \log N(\mathbb{S}_2^{p-1},\log(T) |\cdot|_2,u)du \\
    &\lesssim \int_{0}^{\log (T)} \log \left(\frac{ T}{u}\right) du  +\int_{0}^{\log (T)} \log \left(\frac{ \log(T)}{u}\right)^r du + \int_{0}^{\log (T)} \log \left(\frac{ \log(T)}{u}\right)^p du\\
    &\leq \log^2{(T)} - \int_{0}^{\log T} \log(u)du + ( p +r) \left(\log(T)\log \log(T) - \int_{0}^{\log T} \log(u)du\right)\\
    &\leq \log^2(T) + (p + r) \log(T).
\end{align*}
We now turn to the terms involving the pseudo-distance $|W_{u_1,v_1,w_1} - W_{u_2,v_2,w_2}|_{S_2}$. Again the complex exponential is -Lipschitz, we have 
\begin{align*}
    &d_{S_2} ((u_{1},v_1,w_1),(u_2,v_2,w_2)) \\
    &\leq  \sqrt{\bar{N}} \left(\left(\sum \limits_{l=1}^{2T-1}  \frac{1}{\bar{l}^2} |v_1w_1^*- v_2w_2^* |_{S_2}^2 \right)^{1/2} + \left(\sum \limits_{l=1}^{2T-1}  \frac{1}{\bar{l}^2} | \exp (i2 \pi l u_1) - \exp( i2 \pi l u_2)|^2 \right)^{1/2}\right)\\
    &\lesssim  \sqrt{\bar{N}} \left(|v_1- v_2 |_2 + |w_1- w_2 |_2 + \sqrt{T} |u_1-u_2|\right).
\end{align*}
The radius $\Delta_{S_2}$ satisfies
\begin{equation*}
    \Delta_{S_2} =  2 \sup \limits_{u\in [0\ 1]}|W_u|_{S_2} \simeq \sqrt{\bar{N}}.
\end{equation*}
The $\gamma_2$ functional satisfies
\begin{align*}
    &\gamma_2([0,1] \times \mathbb{S}_2^{r-1} \times \mathbb{S}_2^{p-1},d_2) \leq \gamma_2([0,1], \sqrt{\bar{N}T} |\cdot|)  \\ & \qquad\qquad \qquad \qquad \qquad \qquad \qquad   + \gamma_2(\mathbb{S}_2^{r-1},\sqrt{\bar{N}} |\cdot |_2) + \gamma_2(\mathbb{S}_2^{p-1},\sqrt{\bar{N}} |\cdot |_2)\\
    &\lesssim \int_{0}^{ \Delta_{S_2} } \left( \log N([0,1],\sqrt{\bar{N}T} |\cdot|,u) \right)^{1/2} du + \int_{0}^{ \Delta_{S_2} } \left( \log N(\mathbb{S}_2^{r-1},\sqrt{\bar{N}} |\cdot|_2,u) \right)^{1/2} du  \\ & \qquad + \int_{0}^{ \Delta_{S_2} } \left( \log N(\mathbb{S}_2^{p-1},\sqrt{\bar{N}} |\cdot|_2,u) \right)^{1/2} du  \\
    &= \int_{0}^{\sqrt{\bar{N}}} \left( \log \frac{\sqrt{\bar{N}T}}{u}\right)^{1/2} du + \int_{0}^{\sqrt{\bar{N}}} \left( \log \left(\frac{3\sqrt{\bar{N}}}{u}\right)^r\right)^{1/2} du + \int_{0}^{\sqrt{\bar{N}}} \left( \log \left(\frac{3\sqrt{\bar{N}}}{u}\right)^p\right)^{1/2} du \\
    &=  \sqrt{\bar{N}T}  \int_{\sqrt{\log (T)}}^{\infty}\,\, t^2 \exp(-t^2/2) dt + 6\sqrt{\bar{N}r}  \int_{\sqrt{\log 3}}^{\infty}\,\, t^2 \exp(-t^2/2) dt  \\ & \qquad + 6\sqrt{\bar{N}p}  \int_{\sqrt{\log 3}}^{\infty}\,\, t^2 \exp(-t^2/2) dt\\
    &\lesssim  \sqrt{\bar{N}T}  \left(\frac{\sqrt{\log (T)}}{\sqrt{T}} + \frac{1}{\sqrt{T}\sqrt{\log (T)}}\right) + \sqrt{\bar{N}(p+r)}\lesssim  \sqrt{\bar{N}\log(T)} + \sqrt{\bar{N}(p+r )},
\end{align*}
where in the last step we did and integration by parts and used ~\cite[Formula~$7.1.13$]{Abramowitz:1974:HMF:1098650}.

\noindent Putting these estimates together enables us to bound the supremum of the stochastic polynomial $\chi_{u,v,w}$ with high probability as expressed in \ref{eq: main estimate for noise level}. This in turn implies that with probability at least $1-\exp{(-t)}$, for $t \geq 1$,
\begin{multline*}
       \frac{1}{\bar{N}}|{H^{\dagger}}^{*}X^*\bar{X}\bar{g}|_{S_\infty} \lesssim \sigma_u^2|G|_{S_\infty} \Big( \sqrt{\frac{\log(T) + p + r}{\bar{N}}} \\
       + \frac{\log^2(T) + ( p+ r )\log(T)}{\bar{N}} +\sqrt{ \frac{t}{\bar{N} }} +  \frac{\log(T)t}{\bar{N}} \Big).
\end{multline*}
\end{proof}
\section{Deterministic estimates} \label{Appendix: Deterministic estimates}
In this Appendix we provide the proofs of some deterministic inequalities that are needed especially in Appendix \ref{Appendix: Proofs of the main probabilistic results}.
\begin{proposition}\label{prop: control of nuclear norm by l2 norm}
For $(v_1,w_1)$ and $(v_2,w_2)$ in $\mathbb{S}_2^{r-1}\times \mathbb{S}_2^{p-1}$ the following norm inequality holds:
\begin{equation*}
    |v_1w_1^* - v_2w_2^*|_{S_\infty} \leq |v_1 - v_2|_2 + |w_1 -w_2|_2. 
\end{equation*}
\end{proposition}
\begin{proof}
Take $(v_1,w_1)$ and $(v_2,w_2)$ both in $\mathbb{S}_2^{r-1}\times \mathbb{S}_2^{p-1}$ and note that
    \begin{align*}
        |v_1w_1^* - v_2w_2^*|_{S_\infty}^2 &= \sup \limits_{a \in \mathbb{S}_2^{p-1}} | v_1\innerl{w_1}{a} - v_2\innerl{w_2}{a}|_2^2\\
        &= \sup \limits_{a \in \mathbb{S}_2^{p-1}} | v_1|_2^2\innerl{w_1}{a}^2 + | v_2|_2^2\innerl{w_2}{a}^2 - 2\innerl{v_1}{v_2}\innerl{w_1}{a}\innerl{w_2}{a}\\
        &= \sup \limits_{a \in \mathbb{S}_2^{p-1}} \innerl{w_1}{a}^2 + \innerl{w_2}{a}^2 +(|v_1-v_2|_2^2-2)\innerl{w_1}{a}\innerl{w_2}{a}\\
        &= \sup \limits_{a \in \mathbb{S}_2^{p-1}} \innerl{w_1 - w_2}{a}^2 +|v_1-v_2|_2^2\innerl{w_1}{a}\innerl{w_2}{a}\\
        &\leq \sup \limits_{a \in \mathbb{S}_2^{p-1}} \innerl{w_1 - w_2}{a}^2 +|v_1-v_2|_2^2\\
        &\leq  |v_1 - v_2|_2^2 +|w_1 - w_2|_2^2.
    \end{align*}
Taking the square root we obtain
    \begin{equation*}
        |v_1w_1^* - v_2w_2^*|_{S_\infty} \leq ( |v_1-v_2|_2^2 + |w_1-w_2|_2^2)^{1/2} \leq  |v_1-v_2|_2 + |w_1-w_2|_2.
    \end{equation*}
\end{proof}

\begin{proposition} \label{prop: upper bound on the operator norm of the multiplication process}
Let $\mathcal{H}$ be the infinite bloc Toeplitz matrix made of $2T-1 $ diagolals blocs of the matices $h_l \in \mathcal{M}_{p \times r}(\mathbb{R})$ with $l \in \llbracket 1,2T-1\rrbracket$, then its operator norm $|\mathcal{H}|_{2 \to 2}$ is upper bounded by 
\begin{equation*}
    |\mathcal{H}|_{2 \to 2} \leq \sup \limits_{t \in [0,1]} \left( \left|\sum \limits_{l=1}^{2T-1} \exp(i2 \pi l t)h_{l}    \right|_{S_\infty}\right).
\end{equation*}
\end{proposition}
\begin{proof}
Define the linear operators $\Phi: l_2(\mathbb{Z}) \to L_2(\mathbb{R}^p)$ and $\Psi:l_2(\mathbb{Z}) \to L_2(\mathbb{R}^r)$ such that
$\Phi (u)(t) = \sum \limits_{l=1}^{2T-1} u_l \exp{(i2 \pi l t)}$ and $\Psi (v)(t) = \sum \limits_{l=-\infty}^{+\infty} v_l \exp{(i2 \pi l t)}$. Both are isometries since
\begin{align*}
    |\Phi (u)|_{L_2(\mathbb{R}^p)} & =  \left( \int \limits_{0}^1 \left| \sum \limits_{l=1}^{2T-1} u_l \exp{(i2 \pi l t)}  \right|_2^2 dt\right)^{1/2} \\
    &= \left(  \sum \limits_{l=1}^{2T-1} \sum \limits_{l'=1}^{2T-1} \int \limits_{0}^1 \exp(i2 \pi l t) \exp(-i2 \pi l' t)dt   \innerl{u_l }{ u_{l'} } \right)^{1/2}\\
    &= \left(  \sum \limits_{l=1}^{2T-1} \innerl{u_l }{ u_{l} } \right)^{1/2} = |u|_2^2.
\end{align*}
$\Psi$ is the usual trigonometric isometry. Thus, for $|u|_2 = 1$ we have
\begin{align*}
    |\mathcal{H}u|_2 &= |\Phi\mathcal{H}\Psi^{-1} \Psi u|_{L_2(\mathbb{R}^p)} = |(\Phi\mathcal{H}\Psi^{-1}) (\sum \limits_{l=1}^{2T-1} u_l \exp{(i2 \pi l t)})|_{L_2(\mathbb{R}^p)}\\
    &= \left( \int \limits_{0}^1 \left| \sum \limits_{l=1}^{2T-1} (\Phi\mathcal{H}\Psi^{-1})(u_l) \exp{(i2 \pi l t)}  \right|_2^2 dt\right)^{1/2} \\
    &= \sup \limits_{|w|_2=1}\left( \int \limits_{0}^1 \left| \sum \limits_{l=1}^{2T-1} \sum \limits_{l'=1}^{2T-1} \innerl{h_{l'}^*w }{ u_l } \exp{(i2 \pi l t)}\exp(i2 \pi l' t)  \right|^2 dt \right)^{1/2}\\
    &= \sup \limits_{|w|_2=1} \left( \int \limits_{0}^1 \left| \innerl{ \sum \limits_{l=1}^{2T-1} \exp{(i2 \pi l t)} u_l }{ \sum \limits_{l'=1}^{2T-1} \exp(i2 \pi l' t)h_{l'}^*w }   \right|^2 dt \right)^{1/2}\\
    &\leq \sup \limits_{|w|_2=1} \left( \int \limits_{0}^1 \left|\sum \limits_{l=1}^{2T-1} \exp{(i2 \pi l t)} u_l \right|_2^2 \left|\sum \limits_{l'=1}^{2T-1} \exp(i2 \pi l' t)h_{l'}^*w    \right|_2^2 dt \right)^{1/2}\\
    &\leq \sup \limits_{t \in [0,1]} \left|\sum \limits_{l=1}^{2T-1} \exp(i2 \pi l t)h_{l}    \right|_{S_\infty}\left( \int \limits_{0}^1 \left|\sum \limits_{l=1}^{2T-1} \exp{(i2 \pi l t)} u_l \right|_2^2 dt \right)^{1/2}\\
    &= \sup \limits_{t \in [0,1]} \left( \left|\sum \limits_{l=1}^{2T-1} \exp(i2 \pi l t)h_{l}    \right|_2\right)|u|_2,
\end{align*}
whence, the desired result
\begin{equation*}
    |\mathcal{H}|_{2 \to 2} \leq \sup \limits_{t \in [0,1]} \left( \left|\sum \limits_{l=1}^{2T-1} \exp(i2 \pi l t)h_{l}    \right|_{S_\infty}\right).
\end{equation*}

\end{proof}

\section{Proofs of the remaining results in Theorem \ref{thm: stochatic control of the noise factor}} \label{Appendix: Remaining probabilistic results}
\begin{proof}[Proof of \eqref{eq: second noise term} Theorem \ref{thm: stochatic control of the noise factor}]
The proof is similar to the proof of \eqref{eq: first noise term}. The difference is that $X$ and $W$ are independent and involve different sets of random variables. Recall the definition of the permuted index: 
\begin{equation*}
    \bar{l} = \begin{cases}
    l-2T \quad \text{  if } T+1\leq l\leq 2T-1,\\
l  \qquad \qquad \textrm{otherwise}.
\end{cases}
\end{equation*} 
Define
\begin{align*}
    x &= [u^*_0,\dots,u^*_{N-2},u^*_{N-1}]^* \in \mathbb{R}^{rN}\\
    y &= [w_0^*,\dots,w^*_{N-1}]^* \in \mathbb{R}^{d_0N}\\
    z &= [x^*,y^*]^* \in \mathbb{R}^{(d_0+r)N}\\
    U_l &= [u_{2T-l},u_{2T+1-l},\dots,u_{N-l}]  \in \mathcal{M}_{p  \times \bar{N} }(\mathbb{R}).
\end{align*}
From the definition of $H^{\dagger^*}$ in \textcolor{red}{\ref{Adjoint}} we have
\begin{equation*}
{H^{\dagger^*}}X^*Wh^* = \begin{bmatrix}
&(U_1Wh^*)^*  &\frac{1}{2} (U_2Wh^*)^* &\dots &\frac{1}{T} (U_TWh^*)^* \\
& \frac{1}{2} (U_2Wh^*)^* & \frac{1}{3}(U_3Wh)^* &\dots &\frac{1}{T-1} (U_{T+1}Wh^*)^* \\
&\vdots &\vdots &\vdots &\vdots\\
&\frac{1}{T} (U_TWh^*)^* &\frac{1}{T-1} (U_{T+1}Wh^*)^* &\dots &  (U_{2T-1}Wh^*)^*
\end{bmatrix}.
\end{equation*} 
Define the infinite block Hankel operator $\mathcal{H}: \ l_2(\mathbb{N}) \to l_2(\mathbb{N})$ by the $\mathcal{M}_{p\times r}(\mathbb{R})$ blocks
\begin{align*}
\mathcal{H}_{i,j} = \begin{cases}
1/|\bar{l}|(U_lWh^*)^*  \quad \text{for} \quad (i,j) \in \mathbb{N}^2 \ \text{and} \ 1\leq|i-j|=l \leq 2T-1,\\
0 \quad \text{otherwise}.
\end{cases} 
\end{align*}
Then,
\begin{align*}
    |H^{\dagger^*}X^*Wh^*|_{S_\infty} &\leq \sup \limits_{u \in [0,\ 1]} \left|\sum \limits_{l=1}^{2T-1} \frac{e^{i2 \pi \bar{l} u}}{|\bar{l}|} U_lWh^* \right|_{S_\infty} \\
    &= \sup \limits_{u \in [0\ 1]} \sup \limits_{\substack{|v|_2 = 1\\
    |w|_2 = 1}} \left| \innerl{Wh^*}{\sum \limits_{l=1}^{2T-1} \frac{e^{i2 \pi \bar{l} u}}{|\bar{l}|} U_l^*vw^*} \right|.
\end{align*}
Define, for $l \in [0\ N-2T]$, the vectors $G_l \in \mathcal{M}_{p \times d_0N}(\mathbb{R})$ as
\begin{align*}
        G_l &= [C_0A_0^{2T-1+l},C_0A_0^{2T+l},\dots,C_0,0,\dots,0],
\end{align*}
and for $u \in [0\ 1]$ the matrix valued functions $w_{u,v,w,l} \in \mathcal{M}_{r \times p}(\mathbb{C})$ by $w_{u,v,w,l} = \frac{\exp(i2\pi \bar{l}u)}{|\bar{l}|}vw^*$ for $l \in [1\ 2T-1]$ and the matrix valued functions $W_{u,v,w,l} \in \mathcal{M}_{p \times Nr}(\mathbb{C})$ by
\begin{align*}
W_{u,v,w,k} = \begin{bmatrix}
 & & 0 & w_{u,v,w,2T-1}^* & w_{u,v,w,2T-2}^* &\cdots &w_{u,v,w,1}^* & &
\end{bmatrix} ,    
\end{align*}
with the 1st zero a the $k^{\text{th}}$-position. Define $G$ as $G = [G_0^*,\dots,G_{N-2T}^*]^* $ and $W_{u,v,w}$ as $W_{u,v,w} = [W_{u,v,w,1}^*,\dots,W_{u,v,w,N-2T-1}^*]^*$ satisfying 
\begin{equation*}
    \innerl{\bar{X}\bar{g}^*}{\sum \limits_{l=1}^{2T-1} \frac{e^{i2 \pi \bar{l} u}}{|\bar{l}|} U_l^*vw^*}  = \innerl{W_{u,v,w}x}{Gy}.
\end{equation*}
This gives
\begin{align*}
    |H^{\dagger^*}X^*Wh|_{S_\infty} &\leq  \sup \limits_{u \in [0,\ 1]} \sup \limits_{v \in \mathcal{S}_2^{p-1}}  \left|  \innerl{ \begin{bmatrix} W_{u,v,w} & \\
    & \end{bmatrix} z}{\begin{bmatrix}  &H \\
    & \end{bmatrix} z } \right|, 
\end{align*}
which is the supremum of a second order chaos process defined as follows
\begin{align*}
\chi_{u,v,w} =  \innerl{\begin{bmatrix}  & \\
    H^*W_{u,v,w}& \end{bmatrix} z,}{z}.
\end{align*}
To control the increment of the process we use Hanson Wright inequality which gives us for $t >0$.
\begin{equation*}
    \mathbb{P}\left( |\chi_{u,v,w} - \mathbb{E}(\chi_{u,v,w})|\geq  \sigma_u \sigma_w \left(\sqrt{t} |H^*W_{u,v,w}|_{S_2} + t |H^*W_{u,v,w}|_{S_\infty}  \right) \right) \leq 2 \exp{(-ct)}.
\end{equation*}
Since $ \mathbb{E}(\chi_{u,v,w}) = 0$ we obtain a mixed tail process with probability $1-e^{-t}$
\begin{align*}\begin{array}{ll}
    |\chi_{u_1,v_1,w_1} - \chi_{u_2,v_2,w_2}| \leq  \sigma_u\sigma_w |H|_{S_\infty} \left(\sqrt{t} |W_{u_1,v_1,w_1}-W_{u_2,v_2,w_2}|_{S_2}  \right. \\  \left.  \qquad\qquad\qquad\qquad\qquad\qquad + t |W_{u_1,v_1,w_1}-W_{u_2,v_2,w_2}|_{S_\infty}  \right).
    \end{array}
\end{align*}
The rest of the proof is carried out similar to the proof of \eqref{eq: first noise term} to obtain the following bound that holds with probability at least $1-\exp{(-t)}$, for $t \geq 1$.
\begin{equation*}
       |{H^{\dagger}}^{*}X^*Wh|_{S_\infty} \lesssim \sigma_u\sigma_w|H|_{S_\infty} \left( \sqrt{\frac{\log(T) + p}{N}} + \frac{\log^2(T) + p\log(T)}{N} + \sqrt{\frac{t}{N }} +  \frac{\log(T)t}{N} \right).
\end{equation*}
\end{proof}

\begin{proof}[Proof of \eqref{eq: third noise term} Theorem \ref{thm: stochatic control of the noise factor}]
Again we follow similar steps to the proof of \eqref{eq: first noise term}. We take the following definitions 
\begin{align*}
    x &= [u^*_0,\dots,u^*_{N-2},u^*_{N-1}]^* \in \mathbb{R}^{rN}\\
    \varepsilon &= [v_{2T}^*,\dots,v^*_N]^* \in \mathbb{R}^{p\bar{N}}\\
    y &= [x^*,\varepsilon^*]^* \in \mathbb{R}^{rN+p\bar{N}}\\
    U_l &= [u_{2T-l},u_{2T+1-l},\dots,u_{N-l}]  \in \mathcal{M}_{p  \times \bar{N} }(\mathbb{R}).
\end{align*}
From the definition of $H^{\dagger^*}$ in \textcolor{red}{\ref{Adjoint}} we have
\begin{equation*}
{H^{\dagger^*}}X^*\varepsilon = \begin{bmatrix}
&(U_1\varepsilon)^*  &\frac{1}{2} (U_2\varepsilon)^* &\dots &\frac{1}{T} (U_T\varepsilon)^* \\
& \frac{1}{2} (U_2\varepsilon)^* & \frac{1}{3}(U_3\varepsilon)^* &\dots &\frac{1}{T-1} (U_{T+1}\varepsilon)^* \\
&\vdots &\vdots &\vdots &\vdots\\
&\frac{1}{T} (U_T\varepsilon)^* &\frac{1}{T-1} (U_{T+1}\varepsilon)^* &\dots &  (U_{2T-1}\varepsilon)^*
\end{bmatrix}.
\end{equation*} 
Define the infinite block Hankel operator $\mathcal{J}: \ l_2(\mathbb{N}) \to l_2(\mathbb{N})$ by the $\mathbb{R}^p$ blocks
\begin{align*}
\mathcal{J}_{i,j} = \begin{cases}
1/|\bar{l}|(U_l\varepsilon)^*  \quad \text{for} \quad (i,j) \in \mathbb{N}^2, \ \text{and} \ 1\leq|i-j|=l \leq 2T-1,\\
0 \quad \text{otherwise}.
\end{cases} 
\end{align*}
Then
\begin{align*}
    |H^{\dagger^*}X^*\varepsilon|_{S_\infty} &\leq \sup \limits_{u \in [0,\ 1]} \left|\sum \limits_{l=1}^{2T-1} \frac{e^{i2 \pi \bar{l} u}}{|\bar{l}|} U_l\varepsilon \right|_{S_\infty} \\
    &= \sup \limits_{u \in [0\ 1]} \sup \limits_{\substack{|v|_2 = 1\\
    |w|_2 = 1}} \left| \innerl{\varepsilon}{\sum \limits_{l=1}^{2T-1} \frac{e^{i2 \pi \bar{l} u}}{|\bar{l}|} U_l^*vw^*} \right|
\end{align*}
For $u \in [0\ 1]$ define the matrix valued functions $w_{u,v,w,l} \in \mathcal{M}_{r \times p}(\mathbb{C})$ by $w_{u,v,w,l} = \frac{\exp(i2\pi \bar{l}u)}{|\bar{l}|}vw^*$ for $l \in [1\ 2T-1]$ and the matrix valued functions $W_{u,v,w,l} \in \mathcal{M}_{p \times Nr}(\mathbb{C})$ by
\begin{align*}
W_{u,v,w,k} = \begin{bmatrix}
 & & 0 & w_{u,v,w,2T-1}^* & w_{u,v,w,2T-2}^* &\cdots &w_{u,v,w,1}^* & &
\end{bmatrix} ,    
\end{align*}
with the 1st zero a the $k^{\text{th}}$-position. Put them together in 
$W_{u,v,w}$ since
$$
W_{u,v,w} = [W_{u,v,w,1}^*,\dots,W_{u,v,w,N-2T-1}^*]^*
$$ which satisfy 
\begin{equation*}
    \innerl{\varepsilon}{\sum \limits_{l=1}^{2T-1} \frac{e^{i2 \pi \bar{l} u}}{|\bar{l}|} U_l^*vw^*}  = \innerl{W_{u,v,w}x}{\varepsilon}.
\end{equation*}
This gives
\begin{align*}
    |H^{\dagger^*}X^*\varepsilon|_{S_\infty} &\leq  \sup \limits_{u \in [0,\ 1]} \sup \limits_{v \in \mathcal{S}_2^{p-1}}  \left|  \innerl{ \begin{bmatrix} W_{u,v,w} & \\
    & \end{bmatrix} z}{\begin{bmatrix}  &I_N \\
    & \end{bmatrix} z } \right|, 
\end{align*}
which is the supremum of a second order chaos process defined as follows
\begin{align*}
\chi_{u,v,w} =  \innerl{\begin{bmatrix}  & \\
    W_{u,v,w}& \end{bmatrix} y,}{y}.
\end{align*}
To control the increment of the process we use Hanson Wright inequality which gives us for $t > 0$.
\begin{equation*}
    \mathbb{P}\left( |\chi_{u,v,w} - \mathbb{E}(\chi_{u,v,w})|\geq  {\sigma_v}^2  \left(\sqrt{t} |W_{u,v,w}|_{S_2} + t |W_{u,v,w}|_{S_\infty}  \right) \right) \leq 2 \exp{(-ct)}.
\end{equation*}
Since $ \mathbb{E}(\chi_{u,v,w}) = 0$ we obtain a mixed tail process with probability $1-e^{-t}$
\begin{align*}
    |\chi_{u_1,v_1,w_1} - \chi_{u_2,v_2,w_2}| &\leq  \sigma_v^2 \left(\sqrt{t} |W_{u_1,v_1,w_1}-W_{u_2,v_2,w_2}|_{S_2} + t |W_{u_1,v_1,w_1}-W_{u_2,v_2,w_2}|_{S_\infty}  \right).
\end{align*}
The rest of the proof is carried out similarly to the proof of \eqref{eq: first noise term} to obtain the following bound that holds with probability at least $1-\exp{(-t)}$ for $t \geq 1$:
\begin{equation*}
       |{H^{\dagger}}^{*}X^*\varepsilon|_{S_\infty} \lesssim \sigma_v^2 \left( \sqrt{\frac{\log(T) + p}{N}} + \frac{\log^2(T) + p\log(T)}{N} + \sqrt{\frac{t}{N }} +  \frac{\log(T)t}{N} \right).
\end{equation*}
\end{proof}

\end{document}